\documentclass[preprint,12pt]{elsarticle}

\usepackage{amssymb}
\usepackage{amsmath}
\usepackage{lineno}
\usepackage{epsfig}
\usepackage{geometry}
\usepackage{amssymb,version,graphicx,fancybox,mathrsfs,multirow}
\usepackage{amsfonts,amssymb,amsmath,color,dutchcal,amstext}
\usepackage{amsbsy, amsopn, amscd, amsxtra, amsthm,array}
\usepackage{enumerate,algorithmicx,algorithm}
\usepackage{threeparttable}
\usepackage{algpseudocode}
\usepackage{upref}
\usepackage{caption}
\usepackage{epstopdf}
\usepackage{url,hyperref}
\usepackage{bm}
\usepackage{graphicx}
\usepackage{subfigure,wrapfig}
\usepackage{color,xcolor,stmaryrd}
\usepackage{cases}
\usepackage{mathtools}
\usepackage{extarrows}
\usepackage{bbm}
\usepackage{empheq}
\usepackage{booktabs} 
\usepackage{dsfont}
\usepackage{tikz}
\usetikzlibrary{cd}
\newtheorem{thm}{\bf Theorem}

\newtheorem{lemma}{Lemma}[section]
\newtheorem{definition}{Definition}[section]
\newtheorem{remark}{Remark}[section]
\newtheorem{proposition}{Proposition}[section]

\newenvironment{theorem}{\begin{thm}} {\end{thm}}
\newtheorem{cor}{\bf Corollary}

\theoremstyle{remark}

\theoremstyle{definition}

\numberwithin{table}{section}
\begin{document}

\begin{frontmatter}



\title{A divergence-free projection method for quasiperiodic photonic crystals in three dimensions} 


\author[label1]{Zixuan Gao} 
\author[label1]{Zhenli Xu}
\author[label1]{Zhiguo Yang} 

\address{School of Mathematical Sciences, MOE-LSC and CMA-Shanghai, Shanghai Jiao Tong University, Shanghai 200240, China}
            

\begin{abstract}
This paper presents a point-wise divergence-free projection method for numerical approximations of photonic quasicrystals problems. The original three-dimensional quasiperiodic Maxwell's system is transformed into a periodic one in higher dimensions through a variable substitution involving the projection matrix, such that periodic boundary condition can be readily applied. To deal with the intrinsic divergence-free constraint of the Maxwell's equations, we present a quasiperiodic de Rham complex and its associated commuting diagram, based on which a point-wise divergence-free quasiperiodic Fourier spectral basis is proposed. With the help of this basis, we then propose an efficient solution algorithm for the quasiperiodic source problem and conduct its rigorous error estimate.  Moreover, by analyzing the decay rate of the Fourier coefficients of the eigenfunctions, we further propose a divergence-free reduced projection method for the quasiperiodic Maxwell eigenvalue problem, which significantly alleviates the computational cost. Several numerical experiments are presented to validate the efficiency and accuracy of the proposed method. 
\end{abstract}









\end{frontmatter}



\section{Introduction}
Photonic quasicrystals, which differ fundamentally from periodic lattices in terms of symmetry and order, have led to numerous novel discoveries and applications in physics \cite{rodriguez2008computation,chan1998photonic,zoorob2000complete,dal2003light}. By manipulating the rotation angles and lattice ratios of the quasicrystal, novel localized phenomena as well as the transition between localized and delocalized states can be observed \cite{fu2020optical, gao2023pythagoras,wang2020localization}. These peculiar properties make photonic quasicrystals promising materials, which allow for manipulating optoelectronic properties at will. To predict these properties, it is highly demanding to develop efficient and accurate numerical methods for photonic quasicrystal problems.

The problem explored in this paper is a three-dimensional photonic quasicrystal governed by the Maxwell's equations with quasiperiodic electric permittivity. Despite the seemingly simple form of this system, it is quite challenging numerically. The numerical difficulties can be attributed to two primary issues. The first issue is the unboundedness of the computational domain caused by the quasiperiodic condition. The second issue is the divergence-free condition inherent in the Maxwell's equations. To tackle these obstacles, it is crucial to construct efficient approximation schemes that ensure the intrinsic divergence-free constraint under quasiperiodic condition.



There has been a long standing interest in developing numerical algorithms for quasiperiodic problems. Unlike periodic boundary conditions, quasiperiodic boundary conditions lack decay and translation invariance, which significantly affects the accuracy of direct reduction methods applied to the computational domain. One widely employed method for quasiperiodic problems is the supercell approach \cite{davenport1946simultaneous, goldman1993quasicrystals, steurer1999crystallography}, which is often combined with planewave method \cite{della2006comparative, kaliteevski2000two} or the finite difference time-domain (FDTD) method \cite{kim2005photonic, gauthier2006fdtd}. Despite its popularity, the supercell method exhibits slow convergence and substantial computational costs, especially when dealing with large supercells. Additionally, the simultaneous Diophantine approximation error fails to uniformly decay as the supercell size increases. An alternative approach,  commonly referred to as the projection method (PM)\cite{rodriguez2008computation, jiang2014numerical, jiang2018numerical}, solves the quasiperiodic system by transforming it into a periodic problem in a higher-dimensional lattice. The PM uniformly samples every possible supercell of the infinite quasicrystal, eliminating the need to capture the entire quasiperiodic structure and enabling the use of periodic boundary conditions for domain truncation. Currently, the PM is only limited to the solution of scalar equations such as the Schr\"odinger equation. For quasiperiodic Maxwell's equations, the classical PM fails to work as it cannot avoid spurious eigenmodes resulting from violating the Gauss's law. 
To the best of the authors’ knowledge, there are no available numerical methods which both avoid the Diophantine error and preserve the divergence-free constraint for photonic quasicrystals.

In this study, our focus is on the development of a novel algorithm designed for quasicrystals in three dimensions, which can overcome these numerical challenges. The main contributions of this paper are threefold. Firstly, we derive novel divergence-free Fourier approximaion bases using the quasiperiodic de Rham complex for quasicrystals in three dimensions. Secondly, we propose the projection method based on the divergence-free basis (DF-PM) for quasiperiodic curl-curl problems and conduct a rigorous error estimate to justify the spectral accuracy of it. A highly efficient solution algorithm is then proposed. Thirdly, we propose a novel reduction strategy for the Maxwell eigenvalue problem. Theorem \ref{thm3.3} shows the decay rate of the divergence-free Fourier coefficients. Thanks to this property, we propose the reduced projection method based on the divergence-free basis (DF-RPM) for the quasiperiodic Maxwell eigenvalue problem, as the idea of the RPM \cite{gao2023reduced}. The degree of freedom (DOF) decreases significantly and it is possible to use more bases for higher accuracy. These presented divergence-free projection methods and corresponding
fast solution algorithms are very competitive and computationally attractive.

The rest of this paper is organized as follows. Section \ref{s2} presents the preliminaries on quasiperiodic functions. Section \ref{s3} proposes the DF-PM, by introducing the quasiperiodic de Rham complex, the divergence-free basis space and the projection method applying on the quasiperiodic source problems. It also analyzes the method in detail and proves the error estimation. Section \ref{s4} proposes the DF-RPM, by proving the decay rate of the divergence-free Fourier coefficients and the projection method applying on the quasiperiodic Maxwell eigenvalue problems. Section \ref{s5} demonstrates the numerical results of the divergence-free projection methods for solving quasiperiodic source and eigenvalue problems in three dimensions. Section \ref{s6} concludes the discussions with some closing remarks.

\section{Preliminaries on quasiperiodic functions}\label{s2}
\subsection*{Notation and convention} Let us denote $\mathbb{R}$, $\mathbb{Q}$ and $\mathbb{Z}$ be the sets of real, rational and integer numbers, respectively.  We adopt italic letters for numbers or scalar functions, bold italics for vectors or vector functions, and bold for matrices, respectively.  Let $\Omega$ be a generic domain of interest,  $L^2(\Omega)$ be the space of square-integral  functions  with the inner product and norm denoted by $(\cdot, \cdot)$ and $\| \cdot \|$ as usual, and $H^m(\Omega)$ with $m\geq 0$ be the Sobolev spaces as defined in classic textbooks (see e.g. \cite{adams2003sobolev}). Denote the differential operators  $\nabla \times$ and $\nabla \cdot$ respectively as the usual curl and divergence operators.  Then the $H({\rm curl},\Omega)$ and $H({\rm div},\Omega)$ spaces are respectively defined by
    \begin{equation*}
    H({\rm curl},\Omega)=\big\{\bm v\in L^2(\Omega)^3\big| \nabla\times\bm v\in L^2(\Omega)^3 \big\},\quad  H({\rm div},\Omega)=\big\{\bm v\in L^2(\Omega)^3\big| \nabla\cdot\bm v\in L^2(\Omega) \big\}.
    \end{equation*} 
   
   
    Since there would be frequent conversions between periodic and quasiperiodic functions, we use the subscripts ``per" and ``qp" to differentiate the spaces they belong to, as well as their associated inner products and norms. For instance, given a $d$-dimensional periodic function $F(\bm{z}) \in L^2_{\rm per}([0,T]^d)$ with period $T$ in each dimension (referred to as a ``$T$-periodic function"),  its corresponding inner product and norm are denoted by:
    \begin{equation*}
        (F,G)_{\rm per}=\dfrac{1}{T^d}\int_{[0,T]^d}F\Bar{G}{d}\bm{z},\quad \| F\|_{\rm per}=\sqrt{(F,F)_{\rm per}},
    \end{equation*}
    where $\overline{G}$ represents the complex conjugate of $G \in L^2_{\rm per}([0,T]^d)$.

\subsection*{Quasiperiodic functions}     
    To facilitate the development of efficient and accurate numerical methods for quasiperiodic Maxwell source and eigenvalue problems, we provide a brief exposition of the definition of quasiperiodic functions and their relevant properties (see monographs \cite{bohr2018almost,levitan1982almost} for comprehensive discussions).

    \begin{definition}\label{df1}
        A $d$-dimensional function $f(\bm{z}), \;\bm{z}\in \mathbb{R}^d$ is quasiperiodic if there exists a $d\times n$ projection matrix $\mathbf{P}$ such that $F(\bm{x})=F(\mathbf{P}^\intercal\bm{z})=f(\bm{z})$ is an $n$-dimensional periodic function, where all columns of $\mathbf P$ are linearly independent over $\mathbb{Q}$. $F(\bm x)$ is called the parent function of $f(\bm z)$ with respect to $\mathbf P$. 
    \end{definition}

    \begin{remark}
        It is important to note that the projection matrix $\mathbf{P}$ is not unique. The projection matrix of the quasiperiodic function $f(x)=\cos(x)+\cos(\sqrt{2}x)$ can be chosen as $\mathbf{P}=\big[1,\sqrt{2}\big]$ or $\big[\sqrt{2},1 \big]$. Throughout the paper, $\mathbf{P}$ is chosen such that $F(\bm{x})$ is $2\pi$-periodic and we prescribe $\mathbb T^n=[0,2\pi]^n$ hereafter. 
    \end{remark}

We define the mean value of a $d$-dimensional quasiperiodic function $f(\bm{z})$ as
    \begin{equation}
        m(f)=\lim_{L\rightarrow\infty}\dfrac{1}{|L|^d}\int_{K}f(\bm{z}){d}\bm{z},
    \end{equation}
    where $K=\{\bm{z}\,| \,0\leq |\bm{z}_i|\leq L,i=1,\dots,d\}$. Correspondingly, the square-integral quasiperiodic function space $L^2_{\rm qp}(\mathbb{R}^n)$ is defined as
    \begin{equation}
L^2_{\rm qp}(\mathbb{R}^n):=\{f(\bm z) \,|\;m(f\bar{f})<\infty   \},
    \end{equation}
with the inner product and norm given by
\begin{equation}
(f,g)_{\rm qp}=m(f\bar{g}),\quad \|f \|_{\rm qp}=\sqrt{  (f,f)_{\rm qp}   }\,,
\end{equation}    
where $f$ and $g$ are quasiperiodic functions with respect to the same projection matrix $\mathbf P$. 

It is well-known that $\{ e^{\mathrm{i}\langle\bm{k},\bm{x}\rangle} \}_{ \bm{k}\in\mathbb{Z}^n}$ serves as a complete orthonormal basis for  $L^2_{\rm per}(\mathbb T^n)$ such that for any $F(\bm{x})\in L^2_{\rm per}(\mathbb T^n)$, its has the Fourier series
\begin{equation}\label{dfeq1}
F(\bm{x})=\sum_{\bm{k}\in\mathbb{Z}^n}F_{\bm{k}}e^{\mathrm{i}\langle\bm{k},\bm{x}\rangle},\quad     \hat{F}_{\bm{k}}=\dfrac{1}{(2\pi)^n}\int_{\mathbb T^n}F(\bm{x})e^{-\mathrm{i}\langle\bm{k},\bm{x}\rangle}d\bm{x},
\end{equation}
and there holds the Parseval's equality $ \|F \|^2=\sum_{\bm k\in \mathbb{Z}^n} |F_{\bm{k}}|^2$. Lemma \ref{lm: decay} relates the decay rate of Fourier coefficients with the regularity of a function (see \cite[p.196]{grafakos2008classical}).
\begin{lemma}\label{lm: decay}
Let $m\in\mathbb Z, \;m>0$, suppose $F(\bm{x})\in H^m_{\rm per}(\mathbb [0,T]^n)$, then
\begin{equation}\label{decay}
    |\hat{F}_{\bm k}|\leq (\sqrt{n})^{{m}}T^{-n}|\bm k|^{-m}|F|_{m,\rm per},\quad |F|_{m,\rm per}^2=\sum_{\bm k\in\mathbb Z^d}|\bm k|^{2m}|\hat{F}_{\bm k}|^2.
\end{equation}
\end{lemma}

Quasiperiodic function also has the generalized Fourier series and Parseval's equality. One readily verifies that the set $\big\{e^{\mathrm{i} \langle \bm{q},\bm{z} \rangle} \big \}_{\bm{q}\in \mathbb{R}^d}$  forms a normalized orthogonal system as
\begin{equation}\label{eq: qporth}
\big(e^{\mathrm{i}  \langle \bm{q}_1,\bm{z} \rangle        },e^{\mathrm{i}\langle \bm{q}_2,\bm{z} \rangle} \big)_{\rm qp}=\delta_{\bm{q}_1,\bm{q}_2},\quad \bm{q}_1,\bm{q}_2\in \mathbb{R}^d,
\end{equation}
where $ \delta_{\bm{q}_1,\bm{q}_2}$ is the Dirac delta function.  Thus, one can define the Fourier transform of the quasiperiodic functions, also called the Fourier-Bohr transformation (see \cite{bohr2018almost}), as 
\begin{equation}
    \mathscr{F}_{\rm qp}\{f\}(\bm{q})=m\big(f(\bm z)e^{-\mathrm{i}\langle\bm q,\bm z\rangle}\big),
\end{equation}
Correspondingly, one has the generalized Fourier series of the quasiperiodic function and the Parseval's equality in Lemma \ref{thm2}. 
\begin{lemma}\label{thm2}{(see \cite{bohr2018almost})}
    Any $d$-dimensional quasiperiodic function $f(\bm{z})$ has generalized Fourier series
    \begin{equation}\label{gfs}
        f(\bm{z})= \sum_{\bm q\in\mathbb Z[{\rm col}(\mathbf{P})]}\hat{f}_{\bm q}e^{\mathrm{i}  \langle \bm{q},\bm{z} \rangle},\quad  \hat{f}_{\bm q}=(f,e^{\mathrm{i}  \langle \bm{q},\bm{z} \rangle})_{\rm qp},
    \end{equation}
    where $\hat{f}_{\bm q}$ is called the $\bm q$th generalized Fourier coefficient of $f(\bm z)$ and $ \mathbb Z[{\rm col}(\mathbf{P})]:=\{\bm q \, | \,\bm q=\mathbf{P}\bm k,\;\; \bm k\in  \mathbb{Z}^n     \}$.
     In addition, there holds the Parseval's equality:
      \begin{equation}\label{pareq}
        \|f\|^2_{\rm qp}=\sum_{\bm q\in\mathbb Z[{\rm col}(\mathbf{P})]}|\hat{f}_{\bm q}|^2.
    \end{equation}
\end{lemma}

Correspondingly, one can define the Sobolev spaces $H^m_{\rm qp}(\mathbb R^d),m\in\mathbb Z, m>0$ for quasiperiodic functions
\begin{equation}
    H^m_{\rm qp}(L^2_{\rm per}(\mathbb R^d))=\{f(\bm z)\in \mathbb R^d,\|f(\bm z)\|_{m,\rm qp}<\infty\},
\end{equation}
and its associated norm and semi-norm, 
\begin{equation}
    \|f(\bm z)\|_{m,\rm qp}^2=\sum_{\bm q\in\mathbb Z[{\rm col}(\mathbf{P})]}(1+|\bm q|^{2m})|\hat{f}_{\bm q}|^2, \quad |f(\bm z)|_{m,\rm qp}^2=\sum_{\bm q\in\mathbb Z[{\rm col}(\mathbf{P})]}|\bm q|^{2m}|\hat{f}_{\bm q}|^2.
\end{equation}

Theorem \ref{cosis} relates the generalized coefficients of quasiperiodic functions to the Fourier coefficients of its parent function (\cite{jiang2022numerical}). 
\begin{theorem}\label{cosis}
    Let $f(\bm x)$ be a $d$-dimensional quasiperiodic function. There exists a parent function $F(\bm x)$, which has a one-to-one correspondence between their Fourier coefficients, i.e.
    \begin{equation}\label{paf}
   f(z)=\sum_{\bm q\in\mathbb Z[{\rm col}(\mathbf{P})]}\hat{f}_{\bm q}e^{\mathrm{i}  \langle \bm{q},\bm{z} \rangle},\qquad     F(\bm x)=\sum_{\bm k \in\mathbb Z^n}\hat{F}_{\bm k}e^{{\rm i}\langle\bm k,\bm x\rangle}.
    \end{equation}
    One has
    \begin{equation}
    \quad \hat{F}_{\bm k}=\hat{f}_{\bm q}, \;\; {\rm iff} \;\; \bm q=\mathbf{P}\bm k.
    \end{equation}
\end{theorem}
Thus, one can define the bijective mapping 
\begin{equation}\label{eq: phi}
\phi(\bm{k})=\mathbf{P}\bm{k}=\bm{q}, \;\; \bm{k}\in \mathbb{Z}^{n}, \;\; \bm{q}\in \mathbb{Z}[{\rm col}(\mathbf{P})],
\end{equation}
so as to establish a one-to-one correspondence between the generalized Fourier coefficient of the quasiperiodic function and the Fourier series of the periodic function. Besides, it is necessary to introduce the injective mapping 
\begin{equation}\label{eq: psi}
\mathcal{p}(\bm{z})=\mathbf{P}^\intercal \bm{z}=\bm{x},\;\; \bm{z}\in \mathbb{R}^d,\;\;\bm{x}\in \mathbb{R}^n/\mathbb T^n
\end{equation}
and a functional $\mathcal{P}[\cdot]$ such that one can transform between the quasiperiodic function and its periodic parent function through
\begin{equation}\label{eq: pP}
f(\bm{z})=F\circ    \mathcal{p}(\bm{z}),\quad  \mathcal{P} [f](\bm{x})=F(\bm{x}).
\end{equation}
Given a certain differential operator $D$ w.r.t. variable $\bm{z}$, we denote $D_{\mathcal{p}}$ as the differential operator of $\bm{x}$ obtained by taking variable substitution $\bm{x}=\mathcal{p}(\bm{z})$ and applying the chain rule on $D$. For instance, $\nabla_{  \mathcal{p} }\times$ and $\nabla_{  \mathcal{p} }\cdot$ are respectively given by
\begin{align}
&    \nabla_{ \mathcal{p}}\times\bm F= \sum_{j=1}^n\big(P_{2j}\partial_{x_j} F_3     -P_{3j}\partial_{x_j}F_2,P_{3j}\partial_{x_j}F_1-P_{1j}\partial_{x_j}F_3,P_{1j}\partial_{x_j}F_2-P_{2j}\partial_{x_j}F_1\big)^{\intercal},\\
&     \nabla_{ \mathcal{p}}\cdot \bm F=\sum_{j=1}^n\big(P_{1j} \partial_{x_j}F_1 +P_{2j}\partial_{x_j}F_2+P_{3j}\partial_{x_j}F_3\big),
\end{align}
where $\bm F(\bm x)=(F_1,F_2,F_3)^{\intercal}$. Consequently, one canvreadily obtain the transformation identities
\begin{subequations}
\begin{align}
& D \bm{f}(\bm{z})=D_{\mathcal{p}} \bm{F}\circ \mathcal{p}(\bm{z}),\label{eq: transid2} \\ 
& D_1 D_2 \bm{f}(\bm{z})=(D_{1})_{\mathcal{p}}(D_{2})_{\mathcal{p}} \bm{F}\circ \mathcal{p}(\bm{z}),\label{eq: transid3} \\ 
& \mathcal{P}[Df]=D_{ \mathcal{p}}F, \label{eq: transid1}
\end{align}
\end{subequations}
which will frequently be used in the forthcoming sections. 

In what follows, we are concerned with two typical problems of quasiperiodic photonic crystals.  The first one is the quasiperiodic double-curl source problem (see Eq. \eqref{ccp}), which arises naturally from the time discretization of Maxwell equations with quasiperiodic electric permittivity
\begin{equation}\label{eq001}
    \begin{split}
        \partial_t\bm D-\nabla\times \bm H=-\bm J, &\quad\quad\nabla\cdot \bm D=\bm \rho,\quad{\rm in}\  \mathbb R^3,\\
        \partial_t\bm B+\nabla\times \bm E=\bm 0,&\quad\quad\nabla\cdot \bm B=0,\quad{\rm in}\  \mathbb R^3,
    \end{split}
\end{equation}
where $\bm E, \bm D, \bm B, \bm H,\bm \rho$ and $\bm J$ are electric field, electric flux density, magnetic induction, magnetic field intensity, electric charge density and current density, respectively. Here, $\bm E, \bm D, \bm B, \bm H$ are related by the constitutive relation 
\begin{equation}\label{eq: constutive}
\bm D=\varepsilon \bm E, \quad \bm B=\mu\bm H, 
\end{equation}
where the magnetic permeability  $\mu$ is prescribed as 1 for simplicity, while the electric permittivity $\varepsilon$ is assumed to be quasiperiodic. $\varepsilon(\bm z)$ is a bounded quasiperiodic electric permittivity satisfying $0<\varepsilon_A\leq \varepsilon \leq \varepsilon_B<\infty$. As shown in Figure \ref{fig000}, a three-dimensional quasicrystal that can excite a quasiperiodic dielectric constant is presented.  The other problem is the quasiperiodic double-curl eigenvalue problem (see Eq. \eqref{curlcurl}) obtained from the following Maxwell eigenvalue problem \cite{lyu2021fame,chern2015singular} with quasi-periodic constitutive relation \eqref{eq: constutive}
\begin{equation}\label{eq002}
    \begin{split}
        \nabla\times \bm E={\rm i}\omega \bm B, &\quad\quad\nabla\cdot \bm D=0,\quad{\rm in}\  \mathbb R^3,\\
        \nabla\times \bm H=-{\rm i} {\omega \bm D},&\quad\quad\nabla\cdot \bm B=0,\quad{\rm in}\  \mathbb R^3,
    \end{split}
\end{equation}
which is inevitable when solving the energy states of the quasicrystal photonic crystals. 

\begin{figure}[h]
	\centering
	\includegraphics[width=0.35\textwidth]{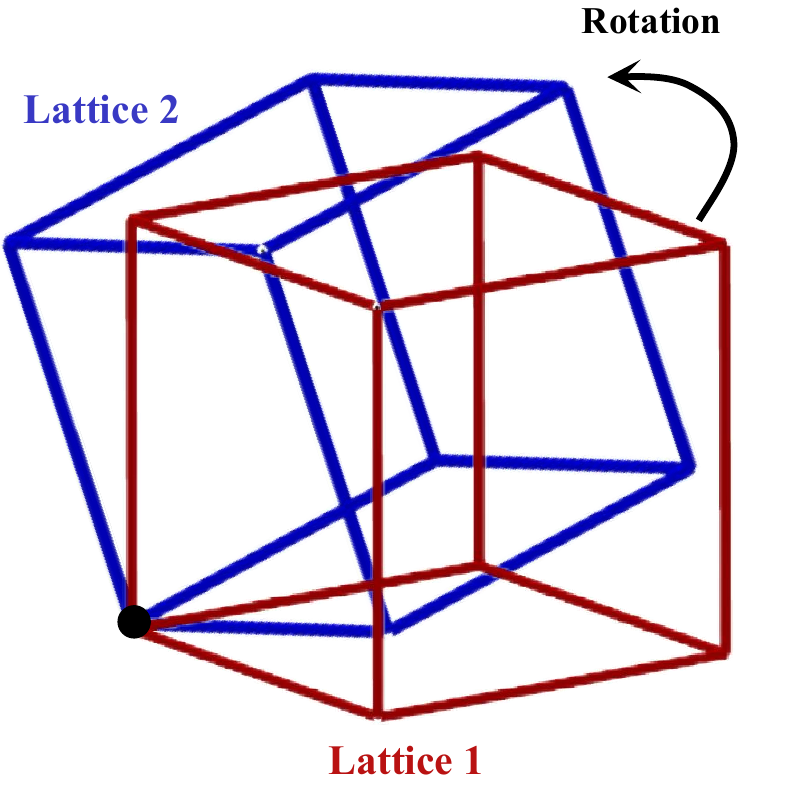}
        \includegraphics[width=0.5\textwidth]{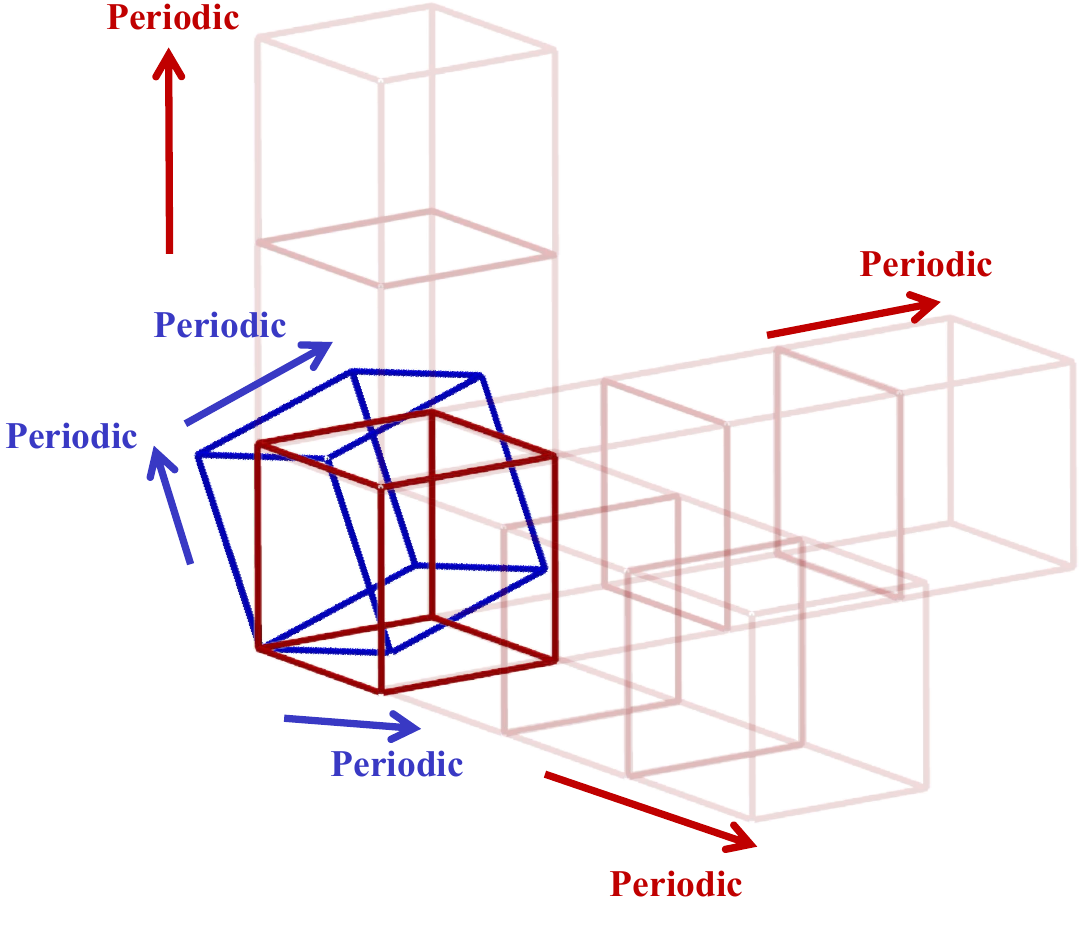}
	\caption{A schematic diagram of a quasicrystal. The red and blue colors represent two identical three-dimensional tetragonal lattices, each exhibiting periodicity in three dimensions. The blue lattice is obtained by rotating the red lattice around a fixed point (black dot) by a certain angle. The superposition of the two lattices forms the quasicrystal.}\label{fig000}
\end{figure}

\section{Divergence-free projection method for quasiperiodic source problem}\label{s3}
In this section, we are devoted to designing numerical methods for the quasiperiodic curl-curl source problem:
\begin{equation}\label{ccp}
    \nabla\times(\varepsilon^{-1}\nabla\times\bm u)+\kappa\bm u=\bm g \;\; {\rm in}\;\;\mathbb R^3,
\end{equation}
where $\bm u(\bm z)$ usually represents the magnetic field satisfying $\nabla\cdot\bm u=0$, $\kappa$ is a given positive constant and $\bm g(\bm z)$ is a body force term satisfying $\nabla\cdot\bm g=0$. We shall propose an accurate and efficient divergence-free projection scheme with the merits of point-wise divergence-free preserving and adhering to de Rham complex. Rigorous error analysis is conducted to validate the accuracy of the proposed method.

\subsection{Quasiperiodic divergence-free basis}\label{ss3}
The first key ingredient of the proposed method to achieve these appealing properties is to construct a point-wise divergence-free quasiperiodic Fourier spectral basis through the following quasiperiodic de Rham complex and its associated commuting diagram (I):
\begin{equation}\label{drc0}
    \begin{tikzcd}[sep=small]
        \mathbb R \rar["Id"] & H_{\rm qp}^1(\mathbb{R}^3) \rar["\nabla"] \dar["\Pi_{N,\rm qp}^1"] & H_{\rm qp}({\rm curl},\mathbb{R}^3) \rar["\nabla\times"] \dar["\Pi_{N,\rm qp}^2"] & H_{\rm qp}({\rm div},\mathbb{R}^3) \rar["\nabla\cdot"] \dar["\Pi_{N,\rm qp}^3"] & L^2_{\rm qp}(\mathbb{R}^3)  \rar["0"]\dar["\Pi_{N,\rm qp}^4"] & 0 \\
        & H_{N,\rm qp}^1(\mathbb{R}^3)\rar["\nabla"]  & H_{N,\rm qp}({\rm curl},\mathbb{R}^3) \rar["\nabla\times"]  & H_{N,\rm qp}({\rm div},\mathbb{R}^3) \rar["\nabla\cdot"]  & L^2_{N,\rm qp}(\mathbb{R}^3)  &   .  \end{tikzcd}
    \end{equation}
 In fact, the bottom row consists of truncated Fourier spaces:
 \begin{align}
     & H_{N,\rm qp}^1(\mathbb{R}^3)=L^2_{N,\rm qp}(\mathbb{R}^3)={\rm span}\big\{ e^{{\rm i}\langle \bm q,\bm z\rangle}\;\big|\; \bm{q}\in  \mathbb{I}_{{\rm qp},N}\big\},\\
     & H_{N,\rm qp}({\rm curl},\mathbb{R}^3)=H_{N,\rm qp}({\rm div},\mathbb{R}^3)={\rm span}\big\{ e^{{\rm i}\langle \bm q,\bm z\rangle}\bm e_j,\;j=1,2,3   \big|\;\bm{q}\in  \mathbb{I}_{{\rm qp},N}  \big\},
 \end{align}
where $\bm e_j$ is the $j$-th unit basis vector and $\mathbb{I}_{{\rm qp},N}$ is the truncated quasiperiodic Fourier index set given by
\begin{equation}\label{eq: indexqp}
\mathbb{I}_{{\rm qp},N}= \big\{   {\bm q\in \mathbb Z[{\rm col}(\mathbf{P})] \big| \;\|\phi^{-1}(\bm q)\|_{\infty}\leq N}   \big\}.
\end{equation}
Correspondingly, the commuting projection operators $\Pi_{N,\rm qp}^j$ are defined as
\begin{align}
&   \Pi_{N,\rm qp}^1 u=\textstyle \sum_{\bm q\in \mathbb{I}_{{\rm qp},N} }\hat{u}_qe^{{\rm i}\langle\bm q,\bm z\rangle},\;\;\;\hat{u}_{\bm q}=(u,e^{\mathrm{i}  \langle \bm{q},\bm{z} \rangle})_{\rm qp}, \;\;\; \forall u\in H_{\rm qp}^1(\mathbb{R}^3), \label{eq:projqp1}\\
& \Pi_{N,\rm qp}^2\bm u=  \textstyle \sum_{\bm q\in \mathbb{I}_{{\rm qp},N} } \hat{\bm u}_qe^{{\rm i}\langle\bm q,\bm z\rangle},\;\;\; \hat{\bm u}_{\bm q}=(\bm{u},e^{\mathrm{i}  \langle \bm{q},\bm{z} \rangle})_{\rm qp},\;\;\; \forall u\in H_{N,\rm qp}({\rm curl},\mathbb{R}^3),\label{eq:projqp2}
\end{align}
and one has $\Pi_{N,\rm qp}^1 u=\Pi_{N,\rm qp}^4 u$ and $\Pi_{N,\rm qp}^2\bm u=\Pi_{N,\rm qp}^3\bm u$.

With the help of the above quasiperiodic de Rham complex, mappings $\mathcal{p}$ and $\mathcal{P}$ in Eq. \eqref{eq: pP} and the transformation identites \eqref{eq: transid2}-\eqref{eq: transid1}, we arrive at the second quasiperiodic complex and its associated diagram (II):
\begin{equation}\label{drc}
    \begin{tikzcd}[sep=small]
        \mathbb R \rar["Id"] & H_{\rm qp}^1(\mathbb{R}^3) \rar["\nabla"] \dar["\mathcal{P}"] & H_{\rm qp}({\rm curl},\mathbb{R}^3) \rar["\nabla\times"] \dar["\mathcal{P}"] & H_{\rm qp}({\rm div},\mathbb{R}^3) \rar["\nabla\cdot"] \dar["\mathcal{P}"] & L^2_{\rm qp}(\mathbb{R}^3)  \rar["0"]\dar["\mathcal{P}"] & 0 \\
        & H_{\rm per}^1 (\mathbb T^n)\rar["\nabla_{\mathcal{p}}"] \dar["\Pi_{N,\rm per}^1"] & H_{\rm per}(\nabla_{\mathcal{p}}\times,\mathbb T^n) \rar["\nabla_{\mathcal{p}}\times"] \dar["\Pi_{N,\rm per}^2"] & H_{\rm per}(\nabla_{\mathcal{p}}\cdot,\mathbb T^n) \rar["\nabla_{\mathcal{p}}\cdot"] \dar["\Pi_{N,\rm per}^3"] & L^2_{\rm per}(\mathbb T^n)  \dar["\Pi_{N,\rm per}^4"] & \\
        & H_{N,\rm per}^1(\mathbb T^n) \rar["\nabla_{\mathcal{p}}"] & H_{N,\rm per}(\nabla_{\mathcal{p}}\times,\mathbb T^n) \rar["\nabla_{\mathcal{p}}\times"] & H_{N,\rm per}(\nabla_{\mathcal{p}}\cdot,\mathbb T^n) \rar["\nabla_{\mathcal{p}}\cdot"] & L^2_{N,\rm per}(\mathbb T^n) & .
    \end{tikzcd}
    \end{equation}
Define the index set $\mathbb{I}_{{\rm per},N}:=\big\{\bm{k}= \mathbb Z^n \big| \|\bm k\|_{\infty}\leq N \big\}$ and the  discrete spaces in the bottom row as follows    
\begin{align}
    &H_{N,\rm per}^1(\mathbb T^n)=L^2_{N,\rm per}(\mathbb T^n)={\rm span}\big\{ e^{{\rm i}\langle\bm k,\bm x\rangle} \;\big| \bm{k}\in \mathbb{I}_{{\rm per},N} \big\},\\
    &H_{N,\rm per}(\nabla_{\mathcal{p}}\times,\mathbb T^n)=H_{N,\rm per}(\nabla_{\mathcal{p}}\cdot,\mathbb T^n)={\rm span}\big\{ e^{{\rm i}\langle\bm k,\bm x\rangle}\bm e_j \;j=1,2,3 \big| \bm{k}\in \mathbb{I}_{{\rm per},N} \big\}.
\end{align}
The corresponding commuting projections can be defined in a similar manner by
\begin{align}
& \Pi_{\rm per}^1\mathcal{P} [u]= \Pi_{\rm per}^4\mathcal{P} [u]=\textstyle \sum_{{\bm k}\in \mathbb{I}_{{\rm per},N}}\hat{U}_{\bm k}e^{{\rm i}\langle\bm k,\bm x\rangle},\;\; \forall \mathcal{P} [u] \in H_{\rm per}^1(\mathbb T^n), \label{eq: projper1}\\
& \Pi_{\rm per}^2 \mathcal{P}[\bm u]=\Pi_{\rm per}^3 \mathcal{P}[\bm u]=\textstyle \sum_{ {\bm k}\in \mathbb{I}_{{\rm per},N} }\hat{\bm U}_{\bm k}e^{{\rm i}\langle\bm k,\bm x\rangle},\;\; \forall \mathcal{P} [\bm u] \in H_{N,\rm per}(\nabla_{\mathcal{p}}\times,\mathbb T^n),\label{eq: projper2}
\end{align}
where we recall theorem \ref{cosis} and Eqs. \eqref{eq:projqp1}-\eqref{eq:projqp2} to obtain that
\begin{equation}
\hat{U}_{\bm k}=\hat{u}_{\bm q}, \;\;\; \hat{\bm U}_{\bm k}=\hat{\bm u}_{\bm q},\;\; {\rm iff} \;\; \bm q=\mathbf{P}\bm k.
\end{equation}
By the transformation identities Eqs.\eqref{eq: transid2}-\eqref{eq: transid1}, it is straightforward to verify that the above diagram is commutable.

Now we are ready to present divergence-free bases for approximating functions in  $H_{\rm qp}({\rm curl},\mathbb{R}^3)$ and $H_{N,\rm per}(\nabla_{\mathcal{p}}\times,\mathbb T^n)$ spaces in Theorem \ref{dfthm}.

\begin{theorem}\label{dfthm}
$\bm{V}_{N,{\rm qp}}^{\rm div0}:=\{\bm u\in H_{N,\rm qp}({\rm curl},\mathbb{R}^3)| \nabla \cdot \bm u=0  \}$ is a divergence-free (w.r.t. the operator $\nabla\cdot $) and $H_{\rm qp}({\rm curl},\mathbb{R}^3)$-conforming approximation space taking the form
\begin{equation}\label{eq: VNdiv0}
\begin{split}
\bm{V}_{N,{\rm qp}}^{\rm div0}&={\rm span}\Big\{ \bm{\psi}_{\bm q}^j(\bm z):={\bm d}_j({\bm q})e^{{\rm i}\langle\bm q,\bm z\rangle},\;j=1,2 \big| \bm{q}\in  \mathbb{I}_{{\rm qp},N} \Big\},
\end{split}
\end{equation}
where $\{{\bm d}_1({\bm q}) ,{\bm d}_2({\bm q})\}$ is an orthonormal basis for the null space of $\bm q.$ The  corresponding divergence-free (w.r.t. the operator $\nabla_{\mathcal{p}}\cdot $) and $H_{\rm per}({\nabla_{\mathcal{p}}\times,\mathbb T^n})$-conforming approximation space $\bm{V}_{N,{\rm per}}^{\rm div0}:=\big \{ \bm U\in H_{N,\rm per}(\nabla_{\mathcal{p}}\times,\mathbb T^n)\big|\nabla_{\rho}\cdot \bm U=0 \big \}$ reads
\begin{equation}\label{eq:Vnper}
\bm{V}_{N,{\rm per}}^{\rm div0}=  {\rm span}\Big\{ \bm{\Psi}_{\bm k}^j(\bm x)=\bm{d}_j(\mathbf{P}\bm k)e^{{\rm i}\langle\bm k,\bm x\rangle},\;j=1,2  \;\big| \bm{k}\in  \mathbb{I}_{{\rm per},N}  \Big\}.
\end{equation}
\end{theorem}

\begin{proof}
Recall from Eq. \eqref{eq:projqp2} that any $\bm{u}\in H_{N,\rm qp}({\rm curl},\mathbb{R}^3)$ can be expanded by $\bm u=\textstyle \sum_{\bm q\in \mathbb{I}_{{\rm qp},N} } \hat{\bm u}_{\bm q}e^{{\rm i}\langle\bm q,\bm z\rangle}$, one can impose the divergence-free condition on $\bm{u}$ and use the orthogonality in Eq. \eqref{eq: qporth} to obtain a sequence of decoupled linear equations
\begin{equation*}
\hat{\bm u}_{\bm q}\cdot \bm q=0, \;\; \bm{q}\in  \mathbb{I}_{{\rm qp},N}.
\end{equation*}
Thus,  if $\bm u$ is divergence-free, then $\hat{\bm u}_{\bm q}$ belongs to the null space of $\bm q=(q_1,q_2,q_3)^{\intercal}$ and can be uniquely determined by a linear combination of a basis $\{{\bm d}_1({\bm q}) ,{\bm d}_2({\bm q})\}$ of this null space as $\hat{\bm u}_{\bm q}=\hat u^{1}_{\bm q} {\bm d}_1({\bm q})+\hat u^{2}_{\bm q} {\bm d}_2({\bm q})$. As a result, we readily obtain 
\begin{equation}
\bm u(\bm z)=\textstyle \sum_{\bm q\in \mathbb{I}_{{\rm qp},N} } \hat{\bm u}_{\bm q}e^{{\rm i}\langle\bm q,\bm z\rangle}=\textstyle \sum_{\bm q\in \mathbb{I}_{{\rm qp},N} }  [\hat u^{1}_{\bm q}  {\bm d}_1({\bm q})+\hat u^{2}_{\bm q}  {\bm d}_2({\bm q})]e^{{\rm i}\langle\bm q,\bm z\rangle},\quad \bm u\in \bm{V}_{N,{\rm qp}}^{\rm div0},
\end{equation}
which leads to the desired basis in Eq. \eqref{eq: VNdiv0}. 


Correspondingly, with the help of the transformation identities Eqs.\eqref{eq: transid2}-\eqref{eq: transid1} and Eq. \eqref{eq: projper2}, one can obtain Eq. \eqref{eq:Vnper} in a similar manner, which ends the proof. 
\end{proof}


\subsection{The divergence-free projection scheme} 

The weak formulation of the quasiperiodic double-curl source problem \eqref{ccp} is to find $\bm u\in\bm V_{\rm qp}^{\rm div0}:=\big\{\bm u\in H_{\rm qp}({\rm curl},\mathbb{R}^3)|\; \nabla \cdot \bm u=0  \big \}$ such that
\begin{equation}\label{ccpw}
    (\varepsilon^{-1}\nabla\times\bm u,\nabla\times\bm v)+\kappa(\bm u,\bm v)=(\bm g,\bm v) \quad \forall \bm v\in \bm V_{\rm qp}^{\rm div0}.
\end{equation}

With the help of the divergence-free Fourier spectral basis for quasiperiodic functions, the proposed point-wise divergence-free projection scheme reads: find $\bm u^N\in\bm V_{N,\rm qp}^{\rm div0}$ such that
\begin{equation}\label{apps}
    (\varepsilon^{-1}\nabla\times\bm u^N,\nabla\times\bm v^N)+\kappa(\bm u^N,\bm v^N)=(\bm g,\bm v^N) \quad \forall \bm v^N\in \bm V_{N,\rm qp}^{\rm div0}.
\end{equation}
By Eqs. \eqref{ccpw} and \eqref{apps}, one defines the corresponding bilinear form for $\bm u,\bm v\in \bm V_{\rm qp}^{\rm div0}$,
\begin{equation}
    \mathcal{A}(\bm u,\bm v)=(\varepsilon^{-1}\nabla\times\bm u,\nabla\times\bm v)+\kappa(\bm u,\bm v),
\end{equation}
and its corresponding periodic form: for 
This bilinear form satisfies the following continuity and coercivity
\begin{align}\label{laxm1}
    &\mathcal{A}(\bm u,\bm v)\leq \max(\varepsilon_A^{-1},\kappa)\|\bm u\|_{\rm qp,curl}\|\bm v\|_{\rm qp,curl},&{\rm for}\;\bm u,\bm v\in \bm V_{\rm qp}^{\rm div0},\\
    &\mathcal{A}(\bm u,\bm u)\geq\min(\varepsilon_B^{-1},\kappa)\|\bm u\|_{\rm qp,curl}^2,&{\rm for}\;\bm u\in \bm V_{\rm qp}^{\rm div0}.\label{laxm2}
\end{align}

\begin{proposition}
    We consider the operator $\pi_{\rm qp}:{\bm V_{\rm qp}^{\rm div0}}\rightarrow\bm V_{N,\rm qp}^{\rm div0}:$ 
\begin{equation}
    \pi_{\rm qp}\bm u= \textstyle\sum_{\bm q\in \mathbb{I}_{{\rm qp},N} } [\hat u^{1}_{\bm q} \bm{\psi}_{\bm q}^1(\bm z)+\hat u^{2}_{\bm q} \bm{\psi}_{\bm q}^2(\bm z)],
\end{equation}
where $\hat u^{j}_{\bm q}=\langle \bm{\hat{u}}_{\bm q},\bm d_j(\bm q)\rangle$, $\bm{\hat{u}}_{\bm q}$ is the Fourier coefficient of $\bm u$ as defined in Eq. \eqref{dfeq1}. It satisfies that 
\begin{equation}\label{projection}
    (\nabla\times(\bm u-\pi_{\rm qp}\bm u),\nabla\times\bm v)+(\bm u-\pi_{\rm qp}\bm u,\bm v)=0,\;\;\forall\bm v\in\bm V_{N,\rm qp}^{\rm div0},
\end{equation}
which impiles that $\pi_{\rm qp}$ is a projection from $\bm V_{\rm qp}^{\rm div0}$ to $\bm V_{N,\rm qp}^{\rm div0}$.
\end{proposition}
\begin{proof}
    Because the Fourier coefficient of $\bm u$ has the property $\hat{\bm u}_{\bm q}=\hat u^{1}_{\bm q} {\bm d}_1({\bm q})+\hat u^{2}_{\bm q} {\bm d}_2({\bm q})$ for all $\bm q\in\mathbb{I}_{{\rm qp},N}$, then by the orthogonality in Eq. \eqref{eq: qporth}, one has 
    \begin{equation}\label{322}
    (\bm u-\pi_{\rm qp}\bm u,\bm v)=0,\;\;\forall\bm v\in\bm V_{N,\rm qp}^{\rm div0}.
    \end{equation}
    Because $\nabla\cdot\bm u=\nabla\cdot(\pi_{\rm qp}\bm u)=0$, by the orthogonality in Eq. \eqref{eq: qporth}, one obtains 
\begin{equation}\label{333}
    (\nabla\times(\bm u-\pi_{\rm qp}\bm u),\nabla\times\bm v)=(-\Delta(\bm u-\pi_{\rm qp}\bm u),\bm v)=0,\;\;\forall\bm v\in\bm V_{N,\rm qp}^{\rm div0}.
\end{equation}
Combining Eqs. \eqref{322} and \eqref{333}, this proposition is proved.
\end{proof}

To derive the error estimate for the proposed scheme, we first introduce the approximation result in Lemma \ref{piqp}. 

\begin{lemma}\label{piqp}
    Suppose that $\bm u$ is a quasiperiodic vector function and define $\bm w=\nabla\times \bm u$. If $\mathcal{P}[\bm u],\mathcal{P}[\bm w]\in H_{\rm per}^m(\mathbb T^n)^3$ with $m\in\mathbb Z^+$, then
    \begin{equation}
        \|\pi_{\rm qp}\bm u-\bm u\|_{\rm qp,curl}\leq N^{-m}(|\mathcal{P}[\bm u]|_{m,\rm per}+|\mathcal{P}[\bm w]|_{m,\rm per}),
    \end{equation}
    where $\|\cdot\|_{\rm qp,curl}$ is the energy norm of $H_{\rm qp}({\rm curl},\mathbb{R}^3)$, which is defined as
    \begin{equation}\label{hcurlnorm}
        \|\bm u\|_{\rm qp,curl}^2=\|\bm u\|_{\rm qp}^2+\|\nabla\times\bm u\|_{\rm qp}^2,\;\;\bm u\in H_{\rm qp}({\rm curl},\mathbb{R}^3).
    \end{equation}
\end{lemma}
\begin{proof}
    By Eqs. \eqref{projection} and \eqref{hcurlnorm}, one obtains
    \begin{equation}
         \begin{split}
             \|\bm u-\pi_{\rm qp}\bm u\|_{\rm qp,curl}^2&=(\nabla\times(\bm u-\pi_{\rm qp}\bm u),\nabla\times(\bm u-\pi_{\rm qp}\bm u))+(\bm u-\pi_{\rm qp}\bm u,\bm u-\pi_{\rm qp}\bm u)\\
             &=(\nabla\times(\bm u-\pi_{\rm qp}\bm u),\nabla\times\bm u)+(\bm u-\pi_{\rm qp}\bm u,\bm u).
         \end{split}
    \end{equation}
    Then because of the orthogonality in Eq. \eqref{eq: qporth}, it gives
    \begin{equation}
        (\nabla\times(\bm u-\pi_{\rm qp}\bm u),\nabla\times\bm u)=(\nabla\times(\bm u-\pi_{\rm qp}\bm u),\bm w-\Pi_{N,\rm qp}^2\bm w)
    \end{equation}
    and
    \begin{equation}
        (\bm u-\pi_{\rm qp}\bm u,\bm u)=(\bm u-\pi_{\rm qp}\bm u,\bm u-\Pi_{N,\rm qp}^2\bm u),
    \end{equation}
    where $\Pi_{N,\rm qp}^2$ is defined in Eq. \eqref{eq:projqp2}. By the fact that
    \begin{equation}
        (\nabla\times(\bm u-\pi_{\rm qp}\bm u),\bm w-\Pi_{N,\rm qp}^2\bm w)\leq \|\bm u-\pi_{\rm qp}\bm u\|_{\rm qp,curl}\|\bm w-\Pi_{N,\rm qp}^2\bm w\|_{\rm qp}
    \end{equation}
    and
    \begin{equation}
        (\bm u-\pi_{\rm qp}\bm u,\bm u-\Pi_{N,\rm qp}^2\bm u)\leq\|\bm u-\pi_{\rm qp}\bm u\|_{\rm qp,curl}\|\bm u-\Pi_{N,\rm qp}^2\bm u\|_{\rm qp},
    \end{equation}
    one obtains that
    \begin{equation}
        \|\bm u-\pi_{\rm qp}\bm u\|_{\rm qp,curl}\leq\|\bm u-\Pi_{N,\rm qp}^2\bm u\|_{\rm qp}+\|\bm w-\Pi_{N,\rm qp}^2\bm w\|_{\rm qp}.
    \end{equation}
    By Parseval's identities for periodic functions and Theorem \ref{cosis}, one has
    \begin{equation}
    \begin{split}
        &\|\Pi_{N,\rm qp}^2\bm u-\bm u\|_{\rm qp}=\|\Pi_{N,\rm per}^2\mathcal{P}[\bm u]-\mathcal{P}[\bm u]\|_{\rm per},\\
        &\|\Pi_{N,\rm qp}^2\bm w-\bm w\|_{\rm qp}=\|\Pi_{N,\rm per}^2\mathcal{P}[\bm w]-\mathcal{P}[\bm w]\|_{\rm per}.
    \end{split}
    \end{equation}
    Then with the help of the projection error estimation result of periodic functions (cf. \cite[p.34]{shen2011spectral}), one readily obtains the desired result.  
\end{proof}

The error estimate of the divergence-free projection scheme \eqref{apps} can now be stated as follows.
\begin{theorem}\label{sourceea}
    Let $\bm u,\bm u^N$ be respectively the solutions of Eqs. \eqref{ccp} and \eqref{ccpw}, and define $\bm w=\nabla\times\bm u$. If $\mathcal{P}[\bm u],\mathcal{P}[\bm w]\in H_{\rm per}^m(\mathbb T^n)^3$ with $m\in\mathbb Z^+$, then we have
    \begin{equation}\label{eq335}
        \|\bm u-\bm u^N\|_{\rm qp,curl}\leq CN^{-m}(|\mathcal{P}[\bm u]|_{m,\rm per}+|\mathcal{P}[\bm w]|_{m,\rm per}),
    \end{equation}
    where the positive constant $C$ only depends on $\kappa,\varepsilon_A$ and $\varepsilon_B$.
\end{theorem}
\begin{proof}
Taking $\bm v=\bm v^N$ in Eq. \eqref{ccpw}, and by Eq. \eqref{apps}, one obtains the error equation
\begin{equation}
    \mathcal{A}(\bm u-\bm u^N,\bm v)=0,\;\;\forall \bm v^N\in\bm V_{N,\rm qp}^{\rm div0},
\end{equation}
which, together with Eqs. \eqref{laxm1} and \eqref{laxm2}, implies that
\begin{equation}
\begin{split}
    \min(\varepsilon_B^{-1},\kappa)\|\bm u-\bm u^N\|_{\rm qp,curl}^2&\leq \mathcal{A}(\bm u-\bm u^N,\bm u-\bm u^N)=\mathcal{A}(\bm u-\bm u^N,\bm u-\bm v^N)\\
    &\leq \max(\varepsilon_A^{-1},\kappa)\|\bm u-\bm u^N\|_{\rm qp,curl}\|\bm u-\bm v^N\|_{\rm qp,curl}.
\end{split}
\end{equation}
Then one has
\begin{equation}
    \|\bm u-\bm u^N\|_{\rm qp,curl}\leq C\inf_{\bm v^N\in\bm V_{N,\rm qp}^{\rm div0}}\|\bm u-\bm v^N\|_{\rm qp,curl}.
\end{equation}
where the positive constant $C=\max(\varepsilon_B,\kappa^{-1}) \max(\varepsilon_A^{-1},\kappa)$. This theorem is proved by taking $\bm v^N=\pi_{\rm qp}\bm u$ and together with Lemma \ref{piqp}.
\end{proof}

\begin{remark}
    It should be noted that under the quasiperiodic condition, the Poincare inequality $\|\bm u\|_{\rm qp}\leq\|\nabla\times\bm u\|_{\rm qp}$ does not hold due to the lack of ellipticity. Thus Eq. \eqref{eq335} cannot be bounded only by $|\mathcal{P}[\bm w]|_{m,\rm per}$, which differs from the usual error estimates of Maxwell problems such as \cite{li2023energy}.
\end{remark}

\subsection{Solution algorithm for quasiperiodic curl-curl source problems}
Next, we describe the solution algorithm as well as the implementation detail of the DF-PM for quasiperiodic curl-curl source problems. 

Consider the uniform mesh $T_h=\{x=mh,\; m=0,1,\cdots,N-1\}^n$ for domain $\mathbb T^n$  with equidistant mesh size $h=2\pi/N$ in each direction, where $N$ is assumed to be an even positive integer for convenience. 

For any $\bm U\in \bm V_{N,\rm qp}^{\rm div0}$ and $\bm F=(\mathcal{P}[\varepsilon])^{-1}\bm U$, define $I_N$ and $I_N^{\rm div0}$ as the trigonometric interpolation operators \cite{shen2011spectral},
\begin{equation}\label{inpn}
    (I_N\bm F)(\bm x)=\textstyle \sum_{\bm k\in\mathbb{I}_{{\rm per},N} }\Tilde{\bm F}_{\bm k}e^{{\rm i}\langle\bm k,\bm x\rangle},\;\; (I_N^{\rm div0}\bm U)(\bm x)=\sum_{\bm k\in\mathbb{I}_{{\rm per},N} }[\tilde{U}^{1}_{\bm k}\bm{\Psi}_{\bm k}^1(\bm x)+\tilde{U}^{2}_{\bm k}\bm{\Psi}_{\bm k}^2(\bm x)],
\end{equation}
where the pseudo-spectral coefficients $\tilde{U}^{j}_{\bm k}$ and $\Tilde{\bm F}_{\bm k}$ are determined such that $(I^{\rm div0}_N\bm U)(\bm x_h)=\bm U(\bm x_h)$ and $(I_N\bm F)(\bm x_h)=\bm F(\bm x_h)$ hold for all $\bm x_h\in T_h$. 

Before we propose the solution algorithm, we first give some necessary notations. $\mathbf I_{n\times m},\mathbf 0_{n\times m}$ and $\mathbf 1_{n\times m}$ represent $n\times m$ identity matrix, zero matrix and all-ones matrix, respectively. We further introduce the notation of ``vec"  to arrange all $U_{\bm k}$ into a column vector $\Vec{\bm U}$, i.e. $\Vec{\bm U}={\rm vec}(U_{\bm k})$, following standard column-wise linear index order of vector $\bm k$ (see \cite[p.27-28]{golub2013matrix} for more details). 

For convenience, we first introduce the definition of the pseudo-spectral divergence-free Fourier coefficient column vector.
\begin{definition}\label{dfFccv}
    For any $\bm U\in \bm V_{N,\rm qp}^{\rm div0}$, its pseudo-spectral divergence-free Fourier coefficient column vector is
    \begin{equation}
        \tilde{\bm U}=\begin{bmatrix}
        \mathbf I_{L\times L}\\\mathbf 0_{L\times L}
    \end{bmatrix}\Vec{{\bm U}}_1+\begin{bmatrix}
        \mathbf 0_{L\times L}\\\mathbf I_{L\times L}
    \end{bmatrix}\Vec{{\bm U}}_2,
    \end{equation}
    where 
    \begin{equation}
        \mathbf I_1=\begin{bmatrix}
        \mathbf I_{L\times L},\mathbf 0_{L\times L}
    \end{bmatrix},\;\;\mathbf I_2=\begin{bmatrix}
        \mathbf 0_{L\times L},\mathbf I_{L\times L}
    \end{bmatrix}
    \end{equation}
    and the $2L\times1$ column vectors are
    \begin{equation}
        \Vec{{\bm U}}_1={\rm vec}(\tilde{U}^{1}_{\bm k}),\;\;\Vec{{\bm U}}_2={\rm vec}(\tilde{U}^{2}_{\bm k}),
    \end{equation}
    where $L$ is the cardinality of $\mathbb{I}_{{\rm per},N}$.
\end{definition}
The vector $\tilde{\bm U}$ contains all the pseudo-spectral divergence-free Fourier coefficients of $\bm U$. One also defines the column vectors
\begin{equation}\label{t}
        \Vec{\bm R}_j={\rm vec}(R^{j}_{\bm k}),\;\;\Vec{\bm S}_j={\rm vec}(S^{j}_{\bm k}),\;\;\vec{\bm T}={\rm vec}(T_{\bm k}),\;\;j=1,2,3,
\end{equation}
where
\begin{equation}
R_{\bm k}^j=\bm e_j^{\intercal}\bm d_1(\mathbf P\bm k),\;\;S_{\bm k}^j=\bm e_j^{\intercal}\bm d_2(\mathbf P\bm k),\;\;T_{\bm k}=\mathrm{i}|\mathbf P\bm k|,
\end{equation}
and $\mathrm{i}$ is the imaginary unit. The two $L\times 3$ matrices $\mathbf R,\mathbf S$ are defined as
\begin{equation}\label{rstw}
    \mathbf{R}=\left(\Vec{\bm R}_1,\Vec{\bm R}_2,\Vec{\bm R}_3\right),\;\;\mathbf{S}=\left(\Vec{\bm S}_1,\Vec{\bm S}_2,\Vec{\bm S}_3\right).
\end{equation}
Then we consider the explicit form of the pseudo-spectral divergence-free Fourier coefficient column vector of $\bm K=\nabla_{\mathcal{p}}\times\bm F$, where $\bm F\in H_{N,\rm per}(\nabla_{\mathcal{p}}\times,\mathbb T^n)$. This form gives the discrete curl operator of the DF-PM, and it is introduced in Theorem \ref{solutional}.
\begin{theorem}\label{solutional}
    For any $\bm F\in H_{N,\rm per}(\nabla_{\mathcal{p}}\times,\mathbb T^n)$, one assumes the $L\times 3$ Fourier coefficient matrix of $\bm F$ is $\hat{\mathbf F}$. The pseudo-spectral divergence-free Fourier coefficient column vector of $\bm K=\nabla_{\mathcal{p}}\times\bm F$ is
    \begin{equation}\label{qk}
        \tilde{\bm K}={\mathbf C}\left(\begin{bmatrix}
        \mathbf I_{L\times L}\\\mathbf 0_{L\times L}
    \end{bmatrix}\hat{\mathbf F}\circ{\mathbf R}\mathbf 1_{3\times 1}+\begin{bmatrix}
        \mathbf 0_{L\times L}\\\mathbf I_{L\times L}
    \end{bmatrix}\hat{\mathbf F}\circ\mathbf{S}\mathbf 1_{3\times 1}\right),
    \end{equation}
    where the notation $\circ$ denotes the Hadamard product between two matrices and the $2L\times2L$ matrix $\mathbf{C}$ is defined as
    \begin{equation}\label{matrixc}
        \mathbf{C}=\begin{bmatrix}
            \mathbf{0}_{L\times L} & \mathbf{I}_{L\times L}\vec{\bm T}\\
            -\mathbf{I}_{L\times L}\vec{\bm T} & \mathbf{0}_{L\times L}
        \end{bmatrix}.
    \end{equation}
    Specially, if $\bm U\in \bm V_{N,\rm per}^{\rm div0}$, and its pseudo-spectral divergence-free Fourier coefficient column vector is $\tilde{\bm U}$, then the pseudo-spectral divergence-free Fourier coefficient column vector of $\bm K=\nabla_{\mathcal{p}}\times\bm U$ is
    \begin{equation}\label{qk2}
        \tilde{\bm K}=\mathbf C\tilde{\bm U}.
    \end{equation}
\end{theorem}
\begin{proof}
    By the difinition of the trigonometric interpolation operators in Eq. \eqref{inpn}, one has
    \begin{equation}\label{ink}
        I_N\bm K=\nabla_{\mathcal{p}}\times(I_N\bm F)=\textstyle \mathrm{i}\sum_{\bm k\in\mathbb{I}_{{\rm per},N} }[(\mathbf P\bm k)\times\Tilde{\bm F}_{\bm k}]e^{{\rm i}\langle\bm k,\bm x\rangle}.
    \end{equation}
    Then by the orthonormality of $\bm d_1(\mathbf P\bm k),\bm d_2(\mathbf P\bm k)$ and $\mathbf P\bm k/|\mathbf P\bm k|$, the Fourier coefficient $\Tilde{\bm F}_{\bm k}$ can be expand by
    \begin{equation}
        \Tilde{\bm F}_{\bm k}=\Tilde{F}_{\bm k}^1\bm d_1(\mathbf P\bm k)+\Tilde{F}_{k}^2\bm d_2(\mathbf P\bm k)+\Tilde{F}_{k}^3\mathbf P\bm k/|\mathbf P\bm k|,
    \end{equation}
    where $\Tilde{F}_{\bm k}^j=\langle\Tilde{\bm F}_{\bm k},\bm d_j(\mathbf P\bm k)\rangle$ for $j=1,2$ and $\Tilde{F}_{\bm k}^3=\langle\Tilde{\bm F}_{\bm k},\mathbf P\bm k/|\mathbf P\bm k|\rangle$. Hence
    \begin{equation}\label{pkfk}
        (\mathbf P\bm k)\times\Tilde{\bm F}_{\bm k}=-\Tilde{F}_{\bm k}^2\bm d_1(\mathbf P\bm k)+\Tilde{F}_{\bm k}^1\bm d_2(\mathbf P\bm k).
    \end{equation}
    Combining Eqs. \eqref{ink} and \eqref{pkfk}, one can derive the form of $\mathbf C$ and Eq. \eqref{qk}. Specially, for $\bm U\in \bm V_{N,\rm per}^{\rm div0}$, one has
    \begin{equation}
        I_N\bm K=\nabla_{\mathcal{p}}\times(I_N\bm U)=\textstyle \mathrm{i}\sum_{\bm k\in\mathbb{I}_{{\rm per},N} }[(\mathbf P\bm k)\times\Tilde{\bm U}_{\bm k}]e^{{\rm i}\langle\bm k,\bm x\rangle},
    \end{equation}
    where $\Tilde{\bm U}_{\bm k}=\Tilde{U}_{\bm k}^1\bm d_1(\mathbf P\bm k)+\Tilde{U}_{k}^2\bm d_2(\mathbf P\bm k)$. Then by the fact that 
    \begin{equation}
        (\mathbf P\bm k)\times\Tilde{\bm U}_{\bm k}=-\Tilde{U}_{\bm k}^2\bm d_1(\mathbf P\bm k)+\Tilde{U}_{\bm k}^1\bm d_2(\mathbf P\bm k),
    \end{equation}
    one obtains Eq. \eqref{qk2}.
\end{proof}

We denote that $\bm F=(\mathcal{P}[\varepsilon])^{-1}\nabla_{\mathcal{p}}\times\bm U,\;\;\bm K=\nabla_{\mathcal{p}}\times\bm F$, where $\bm U\in\bm V_{N,\rm per}^{\rm div0}$. By Theorem \ref{solutional}, the pseudo-spectral coefficient column vector of $\bm K$ is
    \begin{equation}\label{mvp1}
    \bm {\tilde{\bm K}}=\mathbf C\left(\begin{bmatrix}
        \mathbf I_{L\times L}\\\mathbf 0_{L\times L}
    \end{bmatrix}\hat{\mathbf F}\circ\mathbf{R}\mathbf 1_{3\times 1}+\begin{bmatrix}
        \mathbf 0_{L\times L}\\\mathbf I_{L\times L}
    \end{bmatrix}\hat{\mathbf F}\circ\mathbf{S}\mathbf 1_{3\times 1}\right),
\end{equation}
where the $L\times3$ matrix $\hat{\mathbf F}$ is the Fourier coefficient matrix of $\bm F$. $\hat{\mathbf F}$ has the form
\begin{equation}\label{fftifft1}
    \hat{\mathbf F}=\mathcal{F}[(\mathcal{P}[\varepsilon])^{-1}(\bm x)\cdot\mathcal{F}^{-1}(\begin{bmatrix}
        \mathbf I_{L\times L},\mathbf 0_{L\times L}
    \end{bmatrix}{\mathbf C\bm {\tilde{U}}}\mathbf 1_{1\times L}\circ\mathbf R^{\intercal}+\begin{bmatrix}
        \mathbf 0_{L\times L},\mathbf I_{L\times L}
    \end{bmatrix}{\mathbf C\bm {\tilde{\bm U}}}\mathbf 1_{1\times L}\circ\mathbf S^{\intercal})],
\end{equation}
where $\mathcal{F}(\cdot)$ and $\mathcal{F}^{-1}(\cdot)$ denote the $n$-dimensional Fourier transform and inverse Fourier transform.

\begin{remark}\label{matrixh}
    By Eqs. \eqref{mvp1} and \eqref{fftifft1}, the discrete curl-curl operator of the DF-PM can be rewritten into a matrix $\mathbf H$, which is defined in the form of a matrix-vector product $ \bm {\tilde{\bm K}}=\mathbf H\bm {\tilde{\bm U}}$. The size of $\mathbf H$ is $2L\times2L$. Thus the approximate scheme Eq. \eqref{apps} is equivalent to the following matrix equation
\begin{equation}\label{coreequation}
    (\mathbf H+\kappa\mathbf I_{2L\times2L})\bm {\tilde{\bm U}}=\bm {\tilde{\bm G}},
\end{equation}
where $\bm {\tilde{\bm G}}$ is the pseudo-spectral coefficient of $\mathcal{P}[\bm g]$ as defined in Definition \ref{dfFccv}.
\end{remark}

We solve the matrix equation Eq. \eqref{coreequation} by the GMRES solver \cite{saad1986gmres}. We input matrix-vector products instead of forming the entire coefficient matrix, making the calculation more efficient. Algorithm \ref{algorithm0} proposes the details of generating the matrix-vector product of $\mathbf H+\kappa\mathbf I_{2L\times2L}$.

\begin{algorithm}[h]
	\caption{The matrix-vector product generating of the core matrix in the divergence-free projection method for quasiperiodic source problem}
	\label{algorithm0}
    \leftline{{\bf Input:}~3D quasiperiodic electric permittivity $\varepsilon$, projection matrix $\mathbf P$, }
     \leftline{parameter $N$, and $2L\times1$ column, vector $\bm b$} 
	\leftline{{\bf Output:}~($\mathbf H+\kappa\mathbf I_{2L\times2L})\bm b$}
	\begin{algorithmic}[1]
            \State Compute $\Vec{\bm R}_j,\Vec{\bm S}_j,j=1,2,3$ and $\Vec{\bm T}$ by Eq. \eqref{t} and generate matrices $\mathbf R$ and $\mathbf S$ by Eq. \eqref{rstw}
            \State Compute the discrete curl matrix $\mathbf C$ by Eq. \eqref{matrixc}
            \State Compute the $L\times3$ matrix $\hat{\mathbf F}$ with the form
            \begin{equation*}
                \hat{\mathbf F}=\mathcal{F}[(\mathcal{P}[\varepsilon])^{-1}(\bm x)\cdot\mathcal{F}^{-1}(\begin{bmatrix}
        \mathbf I_{L\times L},\mathbf 0_{L\times L}
    \end{bmatrix}{\mathbf C\bm b}\mathbf 1_{1\times L}\circ\mathbf R^{\intercal}+\begin{bmatrix}
        \mathbf 0_{L\times L},\mathbf I_{L\times L}
    \end{bmatrix}{\mathbf C\bm b}\mathbf 1_{1\times L}\circ\mathbf S^{\intercal})]
            \end{equation*}
            \State Compute $\tilde{\bm K}$ by Eq. \eqref{mvp1}
            \State Return $\tilde{\bm K}+\kappa\bm b$
	\end{algorithmic}
\end{algorithm}

\section{Reduced projection method for quasiperiodic Maxwell eigenvalue problem}\label{s4}
In this section, we move to the quasiperiodic Maxwell eigenvalue problem in three dimensions,
\begin{equation}\label{curlcurl}
    \nabla\times\nabla\times (\varepsilon^{-1} \bm u)=\lambda \bm u,\ \ \ \ \nabla\cdot \bm u=0,\quad{\rm in}\  \mathbb R^3.
\end{equation}
The weak formulation of Eq. \eqref{curlcurl} reads: find non-trivial $\bm u\in\bm V_{\rm qp}^{\rm div0}$ and $\lambda\in\mathbb R$ such that
\begin{equation}\label{curlcurlw}
    (\nabla\times(\varepsilon^{-1}\bm u),\nabla\times\bm v)=\lambda(\bm u,\bm v) \quad \forall \bm v\in \bm V_{\rm qp}^{\rm div0}.
\end{equation}
One can use the projection method based on the divergence-free basis to solve this problem numerically, and the approximation is: find non-trivial $\bm u_N\in\bm V_{N,\rm qp}^{\rm div0}$ and $\lambda^N\in\mathbb R$ such that
\begin{equation}\label{curlcurlw2}
    (\nabla\times(\varepsilon^{-1}\bm u^N),\nabla\times\bm v^N)=\lambda^N(\bm u^N,\bm v^N) \quad \forall \bm v^N\in \bm V_{N,\rm qp}^{\rm div0}.
\end{equation}
This section introduces the first numerical algorithm for Maxwell's quasiperiodic problem without spurious modes, applying the projection method while ensuring the divergence-free condition. Before introducing the algorithm, we first focus on a decay property of Fourier coefficients, as it can guide us in further improving the projection method.

\subsection{{Decay rate of the divergence-free Fourier coefficients}}

In this part, we are intended to show that the Fourier series in terms of the divergence-free basis of the eigenvalue problem exhibits a certain decay rate. One sets $\bm u=(u_1,u_2,u_3)$, where $u_1,u_2$ and $u_3$ are 3D aperiodic functions with the same projection matrix $\mathbf P$, and $\bm f=\varepsilon^{-1}\bm u$. Before we prove the main result, we first introduce Lemmas \ref{innerpro} and \ref{sumdecay}.

\begin{lemma}\label{innerpro}
    Let $\bm p,\bm q\in\mathbb Z[{\rm col}(\mathbf P)]$, one has
    \begin{equation}
        |\langle\bm d_j(\bm p),\bm q/|\bm q|\rangle|\leq \min(|\bm p|,|\bm q|)^{-1}|\bm p-\bm q|, j=1,2,
    \end{equation}
where $\{{\bm d}_1({\bm q}) ,{\bm d}_2({\bm q})\}$ is an orthonormal basis for the null space of $\bm q$. 
\end{lemma}

\begin{proof}
    By the orthonormality of $\bm d_j(\cdot)$ and the Schwarz's inequality,
    \begin{equation}
        |\langle\bm d_i(\bm p),\bm q/|\bm q|\rangle|=|\langle\bm d_i(\bm p),\bm p/|\bm p|-\bm q/|\bm q|\rangle|\leq |\bm p/|\bm p|-\bm q/|\bm q||.
    \end{equation}
    Without of generality, one assumes that $|\bm p|\geq |\bm q|$, and denotes $\theta$ as the angle between vectors $\bm p$ and $\bm q$. Then 
    \begin{equation}
        2|\bm q|(|\bm p|-|\bm q|)\cos\theta\leq 2|\bm q|(|\bm p|-|\bm q|)\leq  |\bm p|^2-|\bm q|^2,
    \end{equation}
    which implies that 
    \begin{equation}
        2|\bm q|^2-2|\bm q|^2\cos\theta\leq  |\bm p|^2+|\bm q|^2-2\langle\bm p,\bm q\rangle.
    \end{equation}
    This ends the proof.
\end{proof}

\begin{lemma}\label{sumdecay}
Suppose that a sequence $\{u_{\bm q}\},\bm q\in\mathbb Z[{\rm col}(\mathbf{P})],|u_{\bm q}|\leq 1$ satisfies
\begin{equation}
        |u_{\bm q}|\leq C_{\alpha}|\bm q|^{-\alpha}
\end{equation}
for some $\alpha>1/2$, where $C_{\alpha}$ is a positive constant. Denote $g(\bm z)$ and $G=\mathcal{P}[\bm g]$. Assume $G(\bm x)\in H_{\rm per}^m(\mathbb T^n)$, $(m\in\mathbb{Z}, \,m\geq 0)$. For any $\gamma_1\leq {\rm min}\big\{m-n, \alpha-1/2\big\}$, there exists a constant $C_{\gamma_1}$ such that
\begin{equation}\label{329}
        \sum_{\bm p\in \mathbb Z[{\rm col}(\mathbf P)]}|\hat{g}_{\bm q-\bm p}u_{\bm p}|\leq C_{\gamma_1}|\bm q|^{-\gamma_1},
\end{equation}
where $\hat{g}_{\bm p}$ is the $\bm p$th Fourier coefficient of $g(\bm z)$, $C_{\gamma_1}$ depends on $m,n,\|\mathbf P\|,|G|_{m,\rm per}$ and $\|g\|$. Moreover, if $\nabla g$ is also $L^2$ bounded, then for any $\gamma_2\leq {\rm min}\big\{m-n, \alpha+1/2\big\}$, there exists a constant $C_{\gamma_2}$ such that
\begin{equation}\label{330}
        \sum_{\bm p\in \mathbb Z[{\rm col}(\mathbf P)]}|\hat{g}_{\bm q-\bm k}u_{\bm k}\langle\bm p/|\bm p|,\bm d_j(\bm q)\rangle|\leq C_{\gamma_2}|\bm q|^{-\gamma_2}, \quad j=1,2,
\end{equation}
where $C_{\gamma_2}$ depends on $m,n,\|\mathbf P\|,|G|_{m,\rm per}$ and $\|\nabla g\|$.
\end{lemma}

\begin{proof}
    One has the decomposition of Eq. \eqref{329}
    \begin{equation}\label{de1}
        \sum_{\bm p\in \mathbb Z[{\rm col}(\mathbf P)]}|\hat{g}_{\bm q-\bm p}u_{\bm p}|=\sum_{\bm p\in \Gamma_1}|\hat{g}_{\bm q-\bm p}u_{\bm p}|+\sum_{\bm p\in \Gamma_2}|\hat{g}_{\bm q-\bm p}u_{\bm p}|,
    \end{equation}
    where $\Gamma_1=\{\bm p|\bm p\in\mathbb Z[{\rm col}(\mathbf{P})],|\phi^{-1}(\bm q-\bm p)|>\|\mathbf P\|^{-1}|\bm q|/2\}$ and $\Gamma_2=\mathbb Z[{\rm col}(\mathbf{P})]/\Gamma_1$. We then consider the summations in two parts, separately.

    $\mathbf {Estimate\ of \ the\ \Gamma_1\ part}$. By the fact that $G(\bm x)\in H_{\rm per}^m(\mathbb T^n)$ and Eq. \eqref{decay}, one obtains that
    \begin{equation}
        \left|\hat{G}_{\phi^{-1}(\bm q-\bm p)}\right|\leq \dfrac{n^{m/2}}{(2\pi)^n}\left|\phi^{-1}(\bm q-\bm p)\right|^{-m}|G|_{m,{\rm per}}.
    \end{equation}
Then the bound
     \begin{equation}
        \begin{split}
            \sum_{\bm p\in\Gamma_1}|\hat{g}_{\bm q-\bm p}||u_{\bm p}|&\leq\sum_{\bm p\in\Gamma_1}|\hat{G}_{\phi^{-1}(\bm q-\bm p)}|\leq \dfrac{n^{m/2}}{(2\pi)^n}|G|_{m,{\rm per}} \sum_{\bm p\in\Gamma_1}|\phi^{-1}(\bm q-\bm p)|^{-m}\\
            &=\dfrac{n^{m/2}}{(2\pi)^n}|G|_{m,{\rm per}}\|\mathbf{P}\|^{-1}\sum_{|\bm k|>|\bm q|/2}|\bm k|^{-m}
        \end{split}
     \end{equation}
     holds, where $\bm k\in\mathbb Z^n$. Then by the fact that
     \begin{equation}
        \sum_{|\bm k|>|\bm q|/2}|\bm k|^{-m}\leq\int_{|\bm t|>|\bm q|/2}|\bm t|^{-m}d\bm t=2\pi\cdot2^{n-2}\int_{|\bm q|/2}^{+\infty}|\bm t|^{n-m-1}d|\bm t|=\dfrac{2^{m-1}\pi}{m-n}|\bm q|^{n-m},
    \end{equation}
    one has
    \begin{equation}\label{part1}
        \sum_{\bm p\in\Gamma_1}|\hat{g}_{\bm q-\bm p}||u_{\bm p}|\leq C_1|\bm q|^{n-m},
    \end{equation}
    where $C_1$ is a positive constant depending on $|G|_{m,{\rm per}},\|\mathbf{P}\|^{-1},m$ and $n$.

    $\mathbf {Estimate\ of \ the\ \Gamma_2\ part}$. By the H\"older's inequality with $1/\beta+1/2n=1$, 
    \begin{equation}
    \begin{split}
        \sum_{\bm p\in\Gamma_2}|\hat{g}_{\bm q-\bm p}||u_{\bm p}|&\leq\left(\sum_{\bm p\in\Gamma_2}|\hat{g}_{\bm q-\bm p}|^2\right)^{\frac{1}{\beta}}\left(\sum_{\bm p\in\Gamma_2}|\hat{g}_{\bm q-\bm p}|^{\frac{2\beta n-4n}{\beta}}|u_{\bm p}|^{2n}\right)^{\frac{1}{2n}}\leq \\
        &\max_{\bm p\in \Gamma_2}|\hat{g}_{\bm q-\bm p}|\cdot\|g\|^{\frac{2}{\beta}}\left(\sum_{\bm p\in\Gamma_2}|u_{\bm p}|^{2n}\right)^{\frac{1}{2n}}\leq \|g\|^{\frac{2}{\beta}+\frac{1}{2}}\left(\sum_{\bm p\in\Gamma_2}|u_{\bm p}|^{2n}\right)^{\frac{1}{2n}}.
    \end{split}
    \end{equation}
    For any $\bm p\in\Gamma_2$, one has 
    \begin{equation}
        \|\mathbf{P}\|^{-1}|\bm q|/2\geq|\phi^{-1}(\bm q-\bm p)|\geq \|\mathbf{P}\|^{-1}|\bm q-\bm p|,
    \end{equation}
    hence $|\bm p|>|\bm q|/2$. One can assume that the decay rate of $|u_{\bm q}|$ with $|\bm q|$ is $\alpha$, that is, there exists a positive constant $C_{\alpha}$ independent with $\bm q$ such that $|u_{\bm{q}}|\leq C_{\alpha}|\bm q|^{-\alpha}$. Then there exists a positive constant $C_2$ depending on $n$ and such that
    \begin{equation}\label{eq338}
        \sum_{\bm p\in\Gamma_2}|u_{\bm p}|^{2n}\leq C_{\alpha}\sum_{\bm p\in\Gamma_2}|\bm p|^{-\alpha n}\leq C_{\alpha}|\Gamma_2||\bm q/2|^{-2\alpha n}\leq C_2 |\bm q|^{n-2\alpha n}.
    \end{equation}
    Therefore, there exists a positive constant $C_3$ depending on $n$ and $\|g\|$ such that
    \begin{equation}\label{part2}
        \sum_{\bm p\in\Gamma_2}|\hat{g}_{\bm q-\bm p}||u_{\bm p}|\leq C_3|\bm q|^{-\alpha+1/2}.
    \end{equation}
    Combining Eqs. \eqref{de1},\eqref{part1} and \eqref{part2}, Eq. \eqref{329} is proved.
      On the other hand, Eq. \eqref{330} has the same decomposition as Eq. \eqref{de1}. The $\Gamma_1$ part is bounded as the same way, then there exists a positive constant $C_4$ depending on $|G|_{m,{\rm per}},\|\mathbf{P}\|^{-1},m$ and $n$ such that
      \begin{equation}
          \sum_{\bm p\in\Gamma_1}|\hat{g}_{\bm q-\bm p}||u_{\bm p}\langle\bm p/|\bm p|,\bm d_j(\bm q)\rangle|\leq C_4|\bm q|^{n-m}.
      \end{equation}
For the $\Gamma_2$ part, using Lemma \ref{innerpro}, 
      \begin{equation}
          \sum_{\bm p\in\Gamma_2}|\hat{g}_{\bm q-\bm p}||u_{\bm p}\langle\bm p/|\bm p|,\bm d_j(\bm q)\rangle|\leq |\bm q|^{-1}\sum_{\bm p\in\Gamma_2}|\hat{g}_{\bm q-\bm p}||\bm q-\bm p||u_{\bm p}|,
      \end{equation}
By the H\"older's inequality with $1/\beta+1/2n=1$ and Eq. \eqref{eq338}, one has
      \begin{equation}
          \sum_{\bm p\in\Gamma_2}|\hat{g}_{\bm q-\bm p}||u_{\bm p}\langle\bm p/|\bm p|,\bm d_j(\bm q)\rangle|\leq |\bm q|^{-1}\|\nabla g\|^{\frac{2}{\beta}+\frac{1}{2}}\left(\sum_{\bm p\in\Gamma_2}|u_{\bm p}|^{2n}\right)^{\frac{1}{2n}}\leq C_5|\bm q|^{-\alpha-1/2},
      \end{equation}
where $C_5$ is a positive constant depending on $n$ and $\|\nabla g\|$. Then Eq. \eqref{330} is proved.
\end{proof}

Then we introduce the main result of the decay rate of the generalized Fourier coefficients in Theorem \ref{thm3.3}.

\begin{theorem}\label{thm3.3}
    Let $\bm u(\bm z)$ be the eigenfunction of the quasiperiodic Maxwell eigenvalue problem \eqref{curlcurl} corresponding to $\lambda$, with $\hat{u}^j_{\bm q}$ being its Fourier coefficient under the divergence-free basis, and $\varepsilon(\bm z)$ and $\mathcal{P}[\varepsilon](\bm x)$ being the quasiperiodic electric permittivity function and its parent function, respectively. Assume $\mathcal{P}[\varepsilon]\in H_{\rm per}^m(\mathbb T^n)$, $(m\in\mathbb{Z}, \,m\geq 0)$. For any $\alpha\leq m-n$, there exists a constant $C_{\alpha}$ such that
    \begin{equation}
    |\hat{u}^{j}_{\bm{q}}|\leq C_{\alpha}|\bm q|^{-\alpha},
\end{equation}
where $C_{\alpha}$ depends on $m,n,\|\mathbf P\|,|\mathcal{P}[\varepsilon]|_{m,\rm per},\|\nabla\varepsilon\|,|(\mathcal{P}[\varepsilon])^{-1}|_{m,\rm per}$ and $\|\nabla(\varepsilon^{-1})\|$.
\end{theorem}

\begin{proof}
    Without loss of generality, let us set $\|\bm u\|_{\rm qp}=1$.   Define
    \begin{equation}
        \bm f(\bm{z})=\varepsilon^{-1}(\bm z)\bm u(\bm z)=\sum_{\bm{q}\in\mathbb Z[{\rm col}(\mathbf{P})]}\hat{\bm f}_{\bm{q}}e^{\mathrm{i}\langle \bm{q},\bm{z}\rangle}.
    \end{equation}
    By the Helmhotlz decomposition, there exists $\bm f_{\rm div0}\in\bm V_{\rm qp}^{\rm div0}$ and a function $g$ such that $\bm f=\bm f_{\rm div0}+\nabla g$. Then one has
    \begin{equation}\label{helm}
        \nabla\times\nabla\times\bm f=\nabla\times\nabla\times(\bm f_{\rm div0}+\nabla g)=\nabla\times\nabla\times\bm f_{\rm div0}=-\Delta \bm f_{\rm div0}.
    \end{equation}
    
    Because $\bm f_{\rm div0},\bm u\in\bm V_{\rm qp}^{\rm div0}$, one has their series under the divergence-free basis,
    \begin{equation}
    \bm f_{\rm div0}=\sum_{\bm q\in \mathbb Z[{\rm col}(\mathbf P)]}(\hat{f}^1_{\bm q}\bm{\psi}_{\bm q}^1(\bm z)+\hat{f}^2_{\bm q}\bm{\psi}_{\bm q}^2(\bm z))
    \end{equation}
    and
    \begin{equation}
    \bm u=\sum_{\bm q\in \mathbb Z[{\rm col}(\mathbf P)]}(\hat{u}^1_{\bm q}\bm{\psi}_{\bm q}^1(\bm z)+\hat{u}^2_{\bm q}\bm{\psi}_{\bm q}^2(\bm z)).
    \end{equation}
    It should be noted that $\hat{f}^j_{\bm q}=\langle\hat{\bm f}_{\bm q},\bm d_j(\bm q)\rangle,j=1,2$. For convenience, one defines $\hat{f}^3_{\bm q}=\langle\hat{\bm f}_{\bm q},\bm q/|\bm q|\rangle$. Then by Eq. \eqref{helm} and the orthogonality of  $\{e^{\mathrm{i}\langle \bm{q},\bm{z}\rangle}\}$, one obtains from the weak form that for $j=1,2$,
\begin{equation}\label{344}
        |\bm q|^2\hat{f}^{j}_{\bm q}=\lambda \hat{u}^{j}_{\bm q},\quad \bm q\in \mathbb Z[{\rm col}(\mathbf P)].
    \end{equation}
By the fact that $\|\bm u\|_{\rm qp}=1$, one has $|\hat{u}^{j}_{\bm q}|\leq1$, which implies that
\begin{equation}\label{eq348}
        |\hat{f}^{j}_{\bm q}|\leq \lambda|\bm q|^{-2},\;\;j=1,2.
    \end{equation}
For $\hat{u}^{j}_{\bm q},j=1,2$, one has the expansion
    \begin{equation}
    \begin{split}
        \hat{u}^{j}_{\bm q}&=\langle\hat{\bm u}_{\bm q},\bm d_j(\bm q)\rangle=\sum_{\bm p\in \mathbb Z[{\rm col}(\mathbf P)]}\hat{\varepsilon}_{\bm q-\bm p}\langle\hat{\bm F}_{\bm k},\bm d_j(\bm q)\rangle\\
        &=\sum_{\bm p\in \mathbb Z[{\rm col}(\mathbf P)]}\hat{\varepsilon}_{\bm q-\bm p}(\hat{f}^{1}_{\bm p}\langle\bm d_1(\bm p),\bm d_j(\bm q)\rangle+\hat{f}^{2}_{\bm p}\langle\bm d_2(\bm p),\bm d_j(\bm q)\rangle+\hat{f}^{3}_{\bm p}\langle\bm p/|\bm p|,\bm d_j(\bm q)\rangle),
    \end{split}
    \end{equation}
    where $\hat{\varepsilon}_{\bm p}$ is the $\bm p$th Fourier coefficient of $\varepsilon(\bm z)$. Therefore, for $j=1,2$
    \begin{equation}\label{346}
        |\bm q|^2|\hat{f}^{j}_{\bm q}|\leq \sum_{\bm p\in \mathbb Z[{\rm col}(\mathbf P)]}|\hat{\varepsilon}_{\bm q-\bm p}\hat{f}^{1}_{\bm p}|+|\hat{\varepsilon}_{\bm q-\bm p}\hat{f}^{2}_{\bm p}|+|\hat{\varepsilon}_{\bm q-\bm p}\hat{f}^{3}_{\bm p}\langle\bm p/|\bm p|,\bm d_j(\bm q)\rangle|.
    \end{equation}
 Because of Eq. \eqref{eq348}, one can assume that the decay rate of $|\hat{f}^{j}_{\bm q}|$ is $\alpha,\;j=1,2$, that is, there exists a positive constant $C_{\alpha}$ independent with $\bm q$ and $j$ such that $|\hat{f}^{j}_{\bm q}|\leq C_{\alpha}|\bm q|^{-\alpha}$ for both $j=1$ and $2$. By Lemma \ref{sumdecay}, there exists a positive constant $C_1$ depends on $m,n,\|\mathbf P\|,|\mathcal{E}|_{m,\rm per}$ and $\|\hat{\varepsilon}\|$, such that
\begin{equation}
        \sum_{\bm p\in \mathbb Z[{\rm col}(\mathbf P)]}|\hat{\varepsilon}_{\bm q-\bm p}\hat{f}^{1}_{\bm p}|+|\hat{\varepsilon}_{\bm q-\bm p}\hat{f}^{2}_{\bm p}|\leq C_1|\bm q|^{-\gamma_1},
    \end{equation}
    where $\gamma_1=\min(m-n,\alpha-1/2)$. On the other hand, by $\bm f=\varepsilon^{-1}\bm u$, $\hat{f}^{3}_{\bm q}=\langle\hat{\bm f}_{\bm q},\bm q/|\bm q|\rangle$ has the decomposition
    \begin{equation}
        \begin{split}
            \hat{f}^{3}_{\bm q}&=\sum_{\bm p\in \mathbb Z[{\rm col}(\mathbf P)]}(\hat{\varepsilon}^{-1})_{\bm q-\bm p}\langle\bm u_{\bm p},\bm q/|\bm q|\rangle\\
            &=\sum_{\bm p\in \mathbb Z[{\rm col}(\mathbf P)]}(\hat{\varepsilon}^{-1})_{\bm q-\bm p}[\hat{u}^{1}_{\bm p}\langle \bm d_1(\bm p),\bm q/|\bm q|\rangle+\hat{u}^{2}_{\bm p}\langle \bm d_2(\bm p),\bm q/|\bm q|\rangle],
        \end{split}
    \end{equation}
    where $(\hat{\varepsilon}^{-1})_{\bm q-\bm p}$ denotes the generalized Fourier coefficient of $\hat{\varepsilon}^{-1}(\bm z)$ with frequency $\bm q-\bm p$. Then by Eq. \eqref{344},
    \begin{equation}
        \hat{f}^{3}_{\bm q}=\sum_{\bm p\in \mathbb Z[{\rm col}(\mathbf P)]}(\hat{\varepsilon}^{-1})_{\bm q-\bm p}|\bm p|^2[\hat{f}^{1}_{\bm p}\langle \bm d_1(\bm p),\bm q/|\bm q|\rangle+\hat{f}^{2}_{\bm p}\langle \bm d_2(\bm p),\bm q/|\bm q|\rangle].
    \end{equation}
    By the assumption that $\varepsilon$ is two-sided bounded, $\mathcal{P}[\varepsilon]\in H_{\rm per}^m(\mathbb T^n)$ infers to $\mathcal{E}^{-1}(\bm x)\in H_{\rm per}^m(\mathbb T^n)$. Then by Lemma \ref{sumdecay}, there exists a positive constant $C_2$ depending on $m,n,\|\mathbf P\|$,\\$|(\mathcal{P}[\varepsilon])^{-1}|_{m,\rm per}$ and $\|\nabla(\varepsilon^{-1})\|$, such that
    \begin{equation}
        |\hat{f}^{3}_{\bm q}|\leq C_2|\bm q|^{-\gamma_2}
    \end{equation}
    where $\gamma_2=\min(m-n,\alpha-3/2)$. Then by Lemma \ref{sumdecay}, there exists a positive constant $C_3$ depends on $m,n,\|\mathbf P\|,|\mathcal{P}[\varepsilon]|_{m,\rm per},\|\nabla\varepsilon\|,|(\mathcal{P}[\varepsilon])^{-1}|_{m,\rm per}$ and $\|\nabla(\varepsilon^{-1})\|$, such that
    \begin{equation}
        \sum_{\bm p\in \mathbb Z[{\rm col}(\mathbf P)]}|\hat{\varepsilon}_{\bm q-\bm p}\hat{f}^{3}_{\bm p}\langle\bm p/|\bm p|,\bm d_j(\bm q)\rangle|\leq C_3|\bm q|^{-\gamma_3},
    \end{equation}
    where $\gamma_3=\min(m-n,\gamma_2+1/2)=\min(m-n,\alpha-1)$. Then by Eq. \eqref{346}, one has $\alpha>m-n+2$. By Eq. \eqref{eq348}, $|\hat{u}^{j}_{\bm q}|$ decays with an order of $m-n$, which ends the proof.
\end{proof}

The DOF of the PM on a divergence-free basis is $O(N^n)$, where $n$ is the number of columns of the projection matrix $\mathbf P$. This means that when using Krylov subspace methods \cite{liesen2013krylov,kelley1995iterative,saad2003iterative} to solve for eigenvalues, the required storage space and computation time become prohibitively large, making it infeasible to compute results for larger $N$. This undoubtedly limits the achievable accuracy.

Based on Theorem \ref{thm2}, we can propose a reduced scheme, divergence-free reduced projection method (DF-RPM), for solving quasiperiodic Maxwell eigenvalue problems, which significantly alleviates the ``curse of dimensionality". Compared to directly applying the PM on a divergence-free basis, the DF-RPM significantly reduces the DOF and computational space while maintaining accuracy. This makes it feasible to compute more spectral points for three-dimensional problems. The specific numerical scheme is detailed in the following sections.

\subsection{Divergence-free reduced projection method}
By the fact that the Fourier series in terms of the divergence-free basis of the eigenvalue problem exhibits a certain decay rate, the truncated quasiperiodic Fourier index sets $\mathbb{I}_{{\rm per},R}$ and $\mathbb{I}_{{\rm qp},R}$ are given by
\begin{align}\label{eq: indexqpr}
&\mathbb{I}_{{\rm qp},R}= \big\{{\bm q\in \mathbb Z[{\rm col}(\mathbf{P})] \big| \;\|\phi^{-1}(\bm q)\|_{\infty}\leq N,\|\bm q\|_{\infty}\leq M}   \big\},\\
&\mathbb{I}_{{\rm per},R}=\big\{\bm{k}= \mathbb Z^n \big| \|\bm k\|_{\infty}\leq N,\|\mathbf P\bm k\|_{\infty}\leq M \big\}.
\end{align}
Then the corresponding reduced divergence-free spaces read
\begin{align}
    &\bm{V}_{R,{\rm qp}}^{\rm div0}={\rm span}\big\{ \bm{\psi}_{\bm q}^j(\bm z):={\bm d}_j({\bm q})e^{{\rm i}\langle\bm q,\bm z\rangle},\;j=1,2 \big| \bm{q}\in  \mathbb{I}_{{\rm qp},R} \big\},\\
    &\bm{V}_{R,{\rm per}}^{\rm div0}= {\rm span}\big\{ \bm{\Psi}_{\bm k}^j(\bm x)=\bm{d}_j(\mathbf{P}\bm k)e^{{\rm i}\langle\bm k,\bm x\rangle},\;j=1,2  \;\big| \bm{k}\in  \mathbb{I}_{{\rm per},R}  \big\}.
\end{align}

In what follows, we first give a rigorous estimate of the truncation error of the reduced space for quasiperiodic functions. For a quasiperiodic vector function $\bm f$ and its parent function $\bm F=\mathcal{P}[\bm f]$, the operators $\Pi_{\rm qp}$ and $\Pi_{\rm per}$ defined in Section \ref{ss3} denote the partial sums 
\begin{equation}
    \Pi_{\rm qp}\bm f=\textstyle \sum_{\bm q\in\mathbb{I}_{{\rm qp},N}}\hat{\bm f}_{\bm q}e^{{\rm i}\langle \bm q,\bm z\rangle}=\Pi_{\rm per}\bm F=\sum_{\bm k\in\mathbb{I}_{{\rm per},N} }\hat{\bm F}_{\bm k}e^{{\rm i}\langle \bm k,\bm x\rangle}.
\end{equation}
Similarly we define the corresponding operator for the DF-RPM,
\begin{equation}
   \Pi_{\rm qp}^R\bm f=\textstyle\sum_{\bm q\in\mathbb{I}_{{\rm qp},R}}\hat{\bm f}_{\bm q}e^{{\rm i}\langle \bm q,\bm z\rangle}=\Pi_{\rm per}^R\bm F=\sum_{\bm k\in\mathbb{I}_{{\rm per},R}}\hat{\bm F}_{\bm k}e^{{\rm i}\langle \bm k,\bm x\rangle}.
\end{equation}
Specially, when $\bm f\in\bm V_{\rm qp}^{\rm div0}$,
\begin{equation}
    \Pi_{\rm qp}\bm f=\textstyle\sum_{\bm q\in\mathbb{I}_{{\rm qp},N}}[\hat{f}^{1}_{\bm q}\bm{\psi}_{\bm q}^1(\bm z)+\hat{f}^{2}_{\bm q}\bm{\psi}_{\bm q}^2(\bm z)]=\Pi_{\rm per}\bm F=\sum_{\bm k\in\mathbb{I}_{{\rm per},N}}[\hat{F}^{1}_{\bm k}\bm{\Psi}_{\bm k}^1(\bm x)+\hat{F}^{2}_{\bm k}\bm{\Psi}_{\bm k}^2(\bm x)]
\end{equation}
and
\begin{equation}
    \Pi_{\rm qp}^R\bm f=\textstyle\sum_{\bm q\in\mathbb{I}_{{\rm qp},R}}[\hat{f}^{1}_{\bm q}\bm{\psi}_{\bm q}^1(\bm z)+\hat{f}^{2}_{\bm q}\bm{\psi}_{\bm q}^2(\bm z)]=\Pi^R_{\rm per}\bm F=\sum_{\bm k\in\mathbb{I}_{{\rm per},R}}[\hat{F}^{1}_{\bm k}\bm{\Psi}_{\bm k}^1(\bm x)+\hat{F}^{2}_{\bm k}\bm{\Psi}_{\bm k}^2(\bm x)].
\end{equation}
The truncation error under $L^2$-norm of the operator is bounded in Theorem \ref{l2b}.

\begin{theorem}\label{l2b}
    Suppose that $\bm u$ is a quasiperiodic vector function and $\bm U=\mathcal{P}[\bm u]$. If $\bm U\in H_{\rm per}^m(\mathbb T^n)^3,\bm u\in H_{\rm qp}^{m'}(\mathbb R^3)^3$ with $m,m'\in\mathbb Z^+$ and $0<m\leq m'$, then
    \begin{equation}
        \|\Pi_{\rm qp}^R\bm u-\bm u\|_{\rm qp}\leq M^{-m'}|\bm u|_{m', \rm qp}+N^{-m}|\bm U|_{m, \rm per}.
    \end{equation}
\end{theorem}
\begin{proof}
    By a direct decomposition $\Pi_{\rm qp}^R\bm u-\bm u=\Pi_{\rm qp}^R\bm u-\Pi_{\rm qp}\bm u+\Pi_{\rm qp}\bm u-\bm u$ and the triangle inequality, one has
    \begin{equation}\label{352}
    \begin{split}
            \|\Pi_{\rm qp}^R\bm u-\bm u\|_{\rm qp}&\leq \|\Pi_{\rm qp}^R\bm u-\Pi_{\rm qp}\bm u\|_{\rm qp}+\|\Pi_{\rm qp}\bm u-\bm u\|_{\rm qp}\\
            &=\|\Pi_{\rm qp}^R\bm u-\Pi_{\rm qp}\bm u\|_{\rm qp}+\|\Pi_{\rm per}\bm U-\bm U\|_{\rm per}.
    \end{split}
    \end{equation}
   By the Parseval's equality,
    \begin{equation}\label{353}
    \begin{split}
        \|\Pi_{\rm qp}^R\bm u-\Pi_{\rm qp}\bm u\|_{\rm qp}^2&=\sum_{\bm q\in\mathbb{I}_{{\rm qp},N}/\mathbb{I}_{{\rm qp},R}}|\hat{\bm u}_{\bm q}|^2\leq M^{-2m'}\sum_{\bm q\in\mathbb{I}_{{\rm qp},N}/\mathbb{I}_{{\rm qp},R}}|\bm q|^{2m'}|\hat{\bm u}_{\bm q}|^2\\&\leq M^{-2m'}|\bm u|^2_{m',\rm qp}
        \end{split}
    \end{equation}
    and
    \begin{equation}\label{354}
        \|\Pi_{\rm per}\bm U-\bm U\|_{\rm per}^2=\sum_{\bm k\in\mathbb Z^n/\mathbb{I}_{{\rm qp},N}}|\hat{\bm U}_{\bm k}|^2\leq N^{-2m}\sum_{\bm k\in\mathbb Z^n/\mathbb{I}_{{\rm qp},N}}|\bm k|^{2m}|\hat{\bm U}_{\bm k}|^2\leq N^{-2m}|\bm U|^2_{m,\rm per}.
    \end{equation}
    Combining Eqs. \eqref{352}, \eqref{353} and \eqref{354} ends the proof. 
\end{proof}

\begin{remark}
    If the parent function $F(\bm x)=\mathcal{P}[f]\in H^m_{\rm per}(\mathbb T^n)^3$, one can use Theorem \ref{cosis} to obtain $f(\bm z)\in H^m_{\rm qp}(\mathbb R^3)^3$. This is because
    \begin{equation}
        \|f(\bm z)\|_{m,\rm qp}^2=\sum_{\bm q\in\mathbb Z[{\rm col}(\mathbf{P})]}(1+|\bm q|^{2m})|\hat{f}_{\bm q}|^2=\sum_{\bm k\in\mathbb{Z}^n}(1+|{\mathbf P}\bm k|^{2m})|\hat{F}_{\bm k}|^2\leq C\|F(\bm x)\|_{m,\rm per},
    \end{equation}
    where $C$ is a positive constant depending on $\|\mathbf P\|$. Hence $m\leq m'$ always holds, and the error decay rate of $M$ is faster than that of $N$. This shows the feasibility of a secondary truncation along the direction of $\mathbf{P}\bm k$. This also implies that, if $\bm u$ and $\bm U$ have sufficiently good regularity, the truncation error can achieve exponential decay to both $N$ and $M$. 
\end{remark}

In the spectral-Galerkin method based on the divergence-free basis, the space of the numerical solution and the test basis space are both $\bm V_{R,\rm per}^{\rm div0}$, the approximate scheme for the weak formulation Eq. \eqref{curlcurlw} reads: find non-trivial $\bm U^R\in\bm V_{R,\rm per}^{\rm div0}$ and $\lambda^R\in \mathbb R$ such that
\begin{equation}
    (\nabla_{\mathcal{p}}\times\bm F^R,\nabla_{\mathcal{p}}\times\bm V)=\lambda(\bm U^R,\bm V) \quad \forall \bm V\in \bm V_{R,\rm per}^{\rm div0},
\end{equation}
where $\bm F^R=(\mathcal{P}[\varepsilon])^{-1}\bm U^R$. The matrices $\mathbf R,\mathbf S,\mathbf C,\bm T$ are defined in the same way of Eqs. \eqref{rstw},\eqref{matrixc} and \eqref{t} with $L$ replaced by $L'$. Here $L'$ is the cardinality of $\mathbb{I}_{{\rm qp},R}$.

For any $\bm U\in \bm V_{R,\rm per}^{\rm div0}$, we denote two vector function $\bm F$ and $\bm K$ as
\begin{equation}
    \bm F=(\mathcal{P}[\varepsilon])^{-1}\bm U,\;\;\bm K=\nabla_{\mathcal{p}}\times\nabla_{\mathcal{p}}\times\bm F.
\end{equation}
By Theorem \ref{solutional}, the pseudo-spectral coefficient column vector of $\bm K$ is
\begin{equation}\label{mvp}
    \bm {\tilde{\bm K}}=\mathbf C^2\left(\begin{bmatrix}
        \mathbf I_{L'\times L'}\\\mathbf 0_{L'\times L'}
    \end{bmatrix}\hat{\mathbf F}\circ\mathbf{R}\mathbf 1_{3\times 1}+\begin{bmatrix}
        \mathbf 0_{L'\times L'}\\\mathbf I_{L'\times L'}
    \end{bmatrix}\hat{\mathbf F}\circ\mathbf{S}\mathbf 1_{3\times 1}\right),
\end{equation}
where the $L'\times3$ matrix $\hat{\mathbf F}$ is the Fourier coefficient matrix of $\bm F$. $\hat{\mathbf F}$ has the form
\begin{equation}\label{fftifft}
    \hat{\mathbf F}=\mathcal{F}[(\mathcal{P}[\varepsilon])^{-1}(\bm x)\cdot\mathcal{F}^{-1}(\begin{bmatrix}
        \mathbf I_{L'\times L'},\mathbf 0_{L'\times L'}
    \end{bmatrix}{\bm {\tilde{U}}}\mathbf 1_{1\times L'}\circ\mathbf R^{\intercal}+\begin{bmatrix}
        \mathbf 0_{L'\times L'},\mathbf I_{L'\times L'}
    \end{bmatrix}{\bm {\tilde{\bm U}}}\mathbf 1_{1\times L'}\circ\mathbf S^{\intercal})],
\end{equation}
where $\mathcal{F}(\cdot)$ and $\mathcal{F}^{-1}(\cdot)$ denote the $n$-dimensional Fourier transform and inverse Fourier transform.In this way, the discrete curl-curl operator of the DF-RPM can also be rewritten into a matrix $\mathbf H$ whose size is $2L'\times2L'$. To calculate the eigenvalues of $\mathbf H$, one takes a random starting column vector $\bm b\in\mathbb{R}^{L'\times1}$, and generates the Krylov subspace $K_G=\hbox{span}\{\bm b,\mathbf H\bm b,\cdots,\mathbf H^{G-1}\bm b\}$. The orthonormal basis $\mathbf Q_G=(\bm q_1,\bm q_2,\cdots,\bm q_G)$ of $K_G$ can be generated by the implicitly restarted Arnoldi method \cite{lehoucq2001implicitly}. One can determine the Hessenberg matrix $\mathbf H_G=\mathbf Q_G^{\intercal}\mathbf H\mathbf Q_G=\mathbf Q_G^{\intercal}(\mathbf H\bm q_1,\mathbf H\bm q_2,\cdots,\mathbf H\bm q_G)$ and solve its eigenpairs $\{(\lambda_j,\bm u_j)\}_{j=1}^G$ of $\mathbf H_G$ by the QR algorithm. Detailed procedures of the DF-RPM are summarized in Algorithm \ref{algorithm2}. 

\begin{algorithm}[h]
	\caption{The reduced projection method based on divergence-free bases (DF-RPM) for Maxwell eigenvalue problems}
	\label{algorithm2}
    \leftline{{\bf Input:}~3D quasiperiodic potential $\varepsilon$, projection matrix $\mathbf P$, parameter $N$, $D$, }
    \leftline{$G$, $\delta$ and $\epsilon$} 
	\leftline{{\bf Output:}~Eigenpairs $\{(\lambda_j,\bm u_j)\}_{j=1}^G$}
	\begin{algorithmic}[1]
            \State Determine the index sets $\mathbb{I}_{{\rm per},R}$ and $\mathbb{I}_{{\rm per},N}$ of basis space and take random starting vector $\bm b\in\mathbb{R}^{2L'\times 1}$ 
            \State Compute $\Vec{\bm R}_j,\Vec{\bm S}_j,j=1,2,3$ and $\Vec{\bm T}$ by Eq. \eqref{t} with $L$ replaced by $L'$ and generate matrices $\mathbf R$ and $\mathbf S$ by Eq. \eqref{rstw}
            \State Compute the discrete curl matrix $\mathbf C$ by Eq. \eqref{matrixc} with $L$ replaced by $L'$
            \State Compute the $L'\times3$ matrix $\hat{\mathbf F}$ with the form
            \begin{equation*}
                \hat{\mathbf F}=\mathcal{F}[(\mathcal{P}[\varepsilon])^{-1}(\bm x)\cdot\mathcal{F}^{-1}(\begin{bmatrix}
        \mathbf I_{L'\times L'},\mathbf 0_{L'\times L'}
    \end{bmatrix}{\bm b}\mathbf 1_{1\times L'}\circ\mathbf R^{\intercal}+\begin{bmatrix}
        \mathbf 0_{L'\times L'},\mathbf I_{L'\times L'}
    \end{bmatrix}{\bm b}\mathbf 1_{1\times L'}\circ\mathbf S^{\intercal})]
            \end{equation*}
            \State Generate the matrix-vector product $\mathbf H\bm b$ by
            \begin{equation*}
                \mathbf H\bm b=\bm {\tilde{\bm K}}=\mathbf C^2\left(\begin{bmatrix}
        \mathbf I_{L'\times L'}\\\mathbf 0_{L'\times L'}
    \end{bmatrix}\hat{\mathbf F}\circ\mathbf{R}\mathbf 1_{3\times 1}+\begin{bmatrix}
        \mathbf 0_{L'\times L'}\\\mathbf I_{L'\times L'}
    \end{bmatrix}\hat{\mathbf F}\circ\mathbf{S}\mathbf 1_{3\times 1}\right)
            \end{equation*}
            \State Repeat step 4-5 to calculate $\mathbf H^2\bm b,\cdots,\mathbf H^{G-1}\bm b$ and generate the Krylov subspace $K_G$ and its orthonormal basis $\mathbf Q_G$ 
            \State Determine Hessenberg matrix $\mathbf H_G$ and compute its eigenpairs $\{(\lambda_j,\bm u_j)\}_{j=1}^G$ of $\mathbf H_G$
            \State Project the eigenfunctions $\bm U_j$ back into the three-dimensional space by $\bm u_j$
	\end{algorithmic}
\end{algorithm}

\section{Numerical examples}\label{s5}
We present numerical results to demonstrate the effectiveness of the above methods. For quasiperiodic Maxwell eigenvalue problems, we use a matrix-free preconditioned Krylov subspace method because it requires only the matrix-vector product to be stored at each iteration \cite{stewart2002krylov}. These examples confirm that the method is free of all spurious modes, has exponential accuracy and is advantageous in the CPU time and memory. The calculations presented in this section are executed using Matlab code on an Intel TM core with a clock rate of 3.80 GHz and 32 GB of memory.

\subsection{Example 1}
We first consider the convergence test for curl-curl problem, which has the manufactured solution $\bm u=\nabla\times\bm w$, where
\begin{equation}\label{manso}
\begin{split}
w_1(z_1,z_2,z_3)&={\rm exp}\left(\sin(z_1)\sin(z_2)\sin(z_3)\right),\\
w_2(z_1,z_2,z_3)&={\rm exp}\left(\sin(\sqrt{2}z_1)\sin(\sqrt{2}z_2)\sin(\sqrt{2}z_3)\right),\\
    w_3(z_1,z_2,z_3)&=0.
\end{split}
\end{equation}
The electric permittivity is $\varepsilon=1$, which represents the vacuum case, and $\kappa=100$. The projection matrix is defined as
\begin{equation}\label{prom0}
    \mathbf{P}=\begin{bmatrix}
        1 & 0 & 0 & \sqrt{2} & 0 & 0\\
        0 & 1 & 0 & 0 & \sqrt{2} & 0\\
        0 & 0 & 1 & 0 & 0 & \sqrt{2}\\
    \end{bmatrix}.
\end{equation}
Accordingly, the source term $\bm g$ in Eq. \eqref{ccp} is chosen such that the analytic expression Eq. \eqref{manso} satisfies Eq. \eqref{ccp}.

In Figure \ref{fig5}, we present the discrete $ L^2 $-error and $ L^{\infty} $-error on a semi-logarithmic scale as functions of various $ N $ using the proposed DF-PM within the domain $[-10, 10]^3$. The results demonstrate the expected exponential convergence of the proposed method under both $ L^2 $-norm and $ L^{\infty} $-norm as $ N $ increases. Consequently, the DF-PM serves as an accurate method for addressing the quasiperiodic source problem, in agreement with the discussions in Theorem \ref{sourceea}.

\begin{figure}[h]
	\centering
	\includegraphics[width=0.9\textwidth]{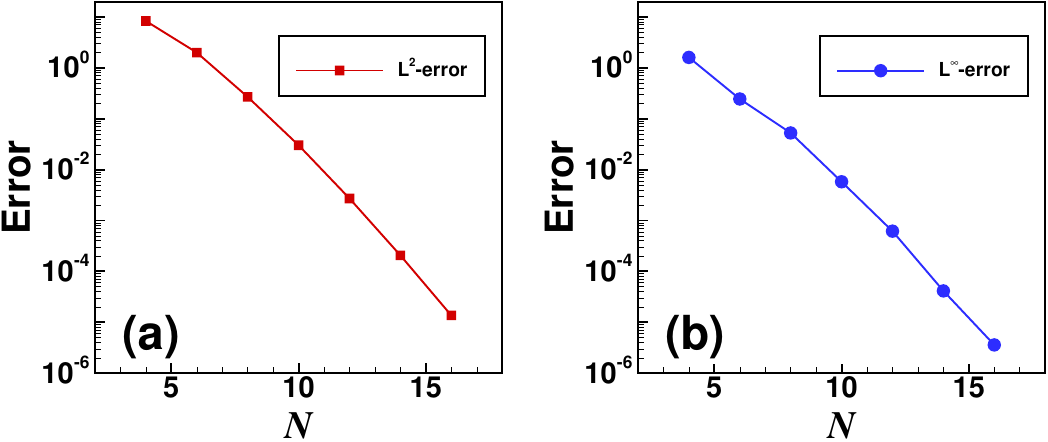}
	\caption{Convergence tests of the curl-curl problem with $\varepsilon=1$. (a) $L^2$-error versus $N$ with $\kappa=100$. (b) $L^{\infty}$-error versus $N$ with $\kappa=100$.}\label{fig5}
\end{figure}

\subsection{Example 2}

We then consider Eq. \eqref{ccp} with a given $\bm g$ with the form $\bm g=\nabla\times\bm w$, where $\bm w$ is defined as Eq. \eqref{manso}. The electric permittivity is
\begin{equation}\label{eps0}
    \varepsilon(z_1,z_2,z_3)=3+\cos(z_1)\cos(z_2)\cos(z_3)+\cos(\sqrt{2}z_1)\cos(\sqrt{2}z_2)\cos(\sqrt{2}z_3),
\end{equation}
and the projection matrix is the same as Eq. \eqref{prom0}. The ``exact" results are determined using numerical results obtained from the DF-PM when $N=16$. The results are shown in Figure \ref{fig6}, and the test area is also $[-10,10]^3$. We observed that both the infinity norm and the 2-norm exhibit exponential decay with increasing $N$, which is consistent with the results of the theoretical error analysis. This means that a smaller $N$ can achieve high accuracy, allowing the three-dimensional curl-curl problem to be computed with fewer degrees of freedom.

\begin{figure}[h]
	\centering
	\includegraphics[width=0.9\textwidth]{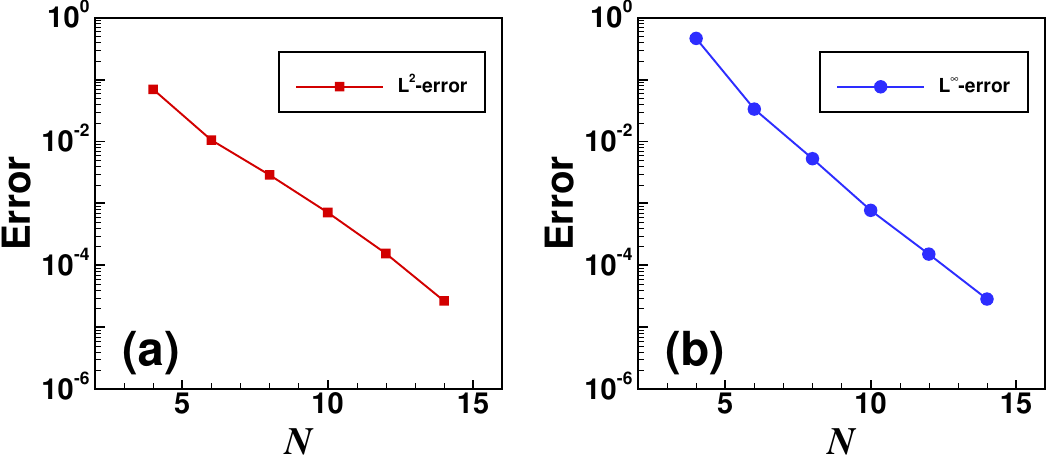}
	\caption{Convergence tests of the curl-curl problem with $\varepsilon$ defined in Eq. \eqref{eps0}. (a) $L^2$-error versus $N$ with $\kappa=1,10,100,1000,10000$. (b) $L^{\infty}$-error versus $N$ with $\kappa=1,10,100,1000,10000$.}\label{fig6}
\end{figure}

\subsection{Example 3}
Then we turn to the quasiperiodic Maxwell eigenvalue problems. We first consider the vacuum case with $\varepsilon=1$. The solution $\bm u$ is an aperiodic vector function with the projection matrix
\begin{equation}\label{prom}
    \mathbf{P}=\begin{bmatrix}
        1 & 0 & 0 & \sqrt{5} & 0 & 0\\
        0 & 1 & 0 & 0 & \sqrt{5} & 0\\
        0 & 0 & 1 & 0 & 0 & \sqrt{5}\\
    \end{bmatrix}.
\end{equation}
The precise Maxwell eigenvalues are determined by $\lambda = p^2 + q^2 + r^2$, where $p, q, r \in \mathbb{Z}[\sqrt{5}]$ and $p^2 + q^2 + r^2 \neq 0$. Consequently, these eigenvalues (spectrum) exhibit density within $\mathbb{R}^+$. We test the first 200 numerical eigenvalues under $N=8,M=6$. As shown in Table \ref{tb022}, all numerical eigenvalues of the DF-RPM have their corresponding real spectrum, which implies that the spurious modes do not exist. The multiplicity of numerical eigenvalues corresponds to the density of states.

\begin{table}[htbp]
	\centering
	\fontsize{10}{10}\selectfont
	\begin{threeparttable}
		\caption{Numerical eigenvalues and their corresponding real spectrum under $N=8,M=6$.} 
		\label{tb022}
  \setlength{\tabcolsep}{8mm}
		\begin{tabular}{ccc}
			\toprule
			\midrule 
                  Numerical eigenvalue & Real spectrum & ID \cr
			\midrule
                 0.0557281 & $(\sqrt{5}-2)^2$ & 1-12 \\
                 0.1114562 & $(\sqrt{5}-2)^2+(\sqrt{5}-2)^2$ & 13-36 \\
                 0.1671843 & $(\sqrt{5}-2)^2+(\sqrt{5}-2)^2+(\sqrt{5}-2)^2$ & 37-52 \\
                0.2229124 & $(2\sqrt{5}-4)^2$ & 53-58\\
                0.2786405 & $(2\sqrt{5}-5)^2$ & 59-82\\
                 0.3343686 & $(2\sqrt{5}-4)^2+(\sqrt{5}-2)^2$ & 83-106\\
                0.4458247 & $(2\sqrt{5}-4)^2+(2\sqrt{5}-4)^2$& 107-112 \\
                 0.5015528 & $(2\sqrt{5}-4)^2+(2\sqrt{5}-4)^2+(\sqrt{5}-2)^2$& 113-124 \\
                 0.5835921 & $(\sqrt{5}-3)^2$ & 125-136 \\
                 0.6393202 & $(\sqrt{5}-3)^2+(\sqrt{5}-2)^2$ & 137-160\\
                 0.6687371 & $(2\sqrt{5}-4)^2+(2\sqrt{5}-4)^2+(2\sqrt{5}-4)^2$ & 161-162 \\
                 0.6950483 & $(\sqrt{5}-3)^2+(\sqrt{5}-2)^2+(\sqrt{5}-2)^2$ & 163-185\\
                0.8065045 & $(\sqrt{5}-3)^2+(2\sqrt{5}-4)^2$ & 185-200+ \\
			\midrule
			\bottomrule
		\end{tabular}
	\end{threeparttable}
\end{table}

In addition, as it is shown in Figure \ref{fig2}, the absolute errors of the first 200 numerical eigenvalues are consistently below $10^{-12}$ in both cases, even under a small $N$ and $M$, affirming the method's correctness and high accuracy.

\begin{figure}[t!]
	\centering
	\includegraphics[width=0.9\textwidth]{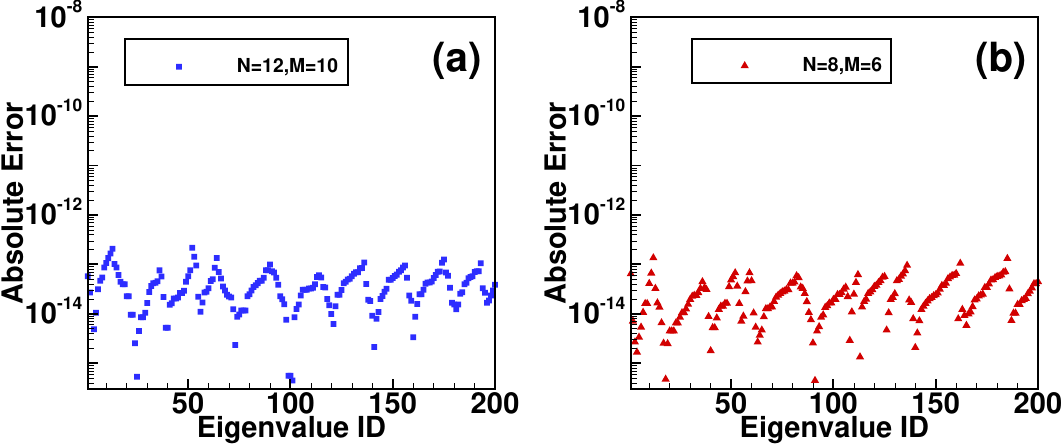}
	\caption{Absolute error versus the eigenvalue ID under (a) $N=12,M=10$ and (b) $N=8,M=6$.}\label{fig2}
\end{figure}

\subsection{Example 4}
We then consider another more complex aperiodic case with 
\begin{equation}
    \varepsilon(\bm z)=1/[10+\cos(z_1)+\cos(z_2)+\cos(z_3)+\cos(\sqrt{5}z_1)+\cos(\sqrt{5}z_2)+\cos(\sqrt{5}z_3)]
\end{equation}
and the projection matrix as Eq. \eqref{prom}. 

To showcase the exceptional accuracy and rapid convergence of the DF-PM and DF-RPM approach, we present errors in Table \ref{tb0} and Figure \ref{fig331}. The ``exact" results are determined using numerical results obtained from the DF-RPM when $N=14$ and $M=10$. The specturm errors, denoted as $\varepsilon_1$ and $\varepsilon_2$, are measured by the maximum error of two spectrum $0.0540$ and $1.0886$. We first consider the DF-PM. To prove the convergence of the DF-PM, $\varepsilon_1$ and $\varepsilon_2$ are included in Table \ref{tb0} to verify the convergence rate. It can be easily found that the error decays rapidly with an increasing $N$, clearly illustrating the correctness of the DF-PM.

\begin{table}[htbp]
	\centering
	\fontsize{10}{10}\selectfont
	\begin{threeparttable}
		\caption{$\varepsilon_1$ and $\varepsilon_2$ of the DF-PM for different $N$ in Example 2} 
		\label{tb0}
  \setlength{\tabcolsep}{12mm}
		\begin{tabular}{ccc}
			\toprule
			\midrule 
                  $N$ & $\varepsilon_1$ & $\varepsilon_2$ \cr
			\midrule
                4 & $1.08\times10^{-4}$    & $1.02\times10^{-3}$  \cr
   			6 & $6.89\times10^{-5}$    & $8.13\times10^{-6}$  \cr
      		8 &  $6.58\times10^{-8}$  & $6.96\times10^{-7}$  \cr
         	10 &  $2.74\times10^{-10}$   & $1.13\times10^{-8}$  \cr
                12 &  $6.21\times10^{-12}$ & $4.50\times10^{-9}$ \cr
			\midrule
			\bottomrule
		\end{tabular}
	\end{threeparttable}
\end{table}

\begin{figure}[h]
	\centering
	\includegraphics[width=0.9\textwidth]{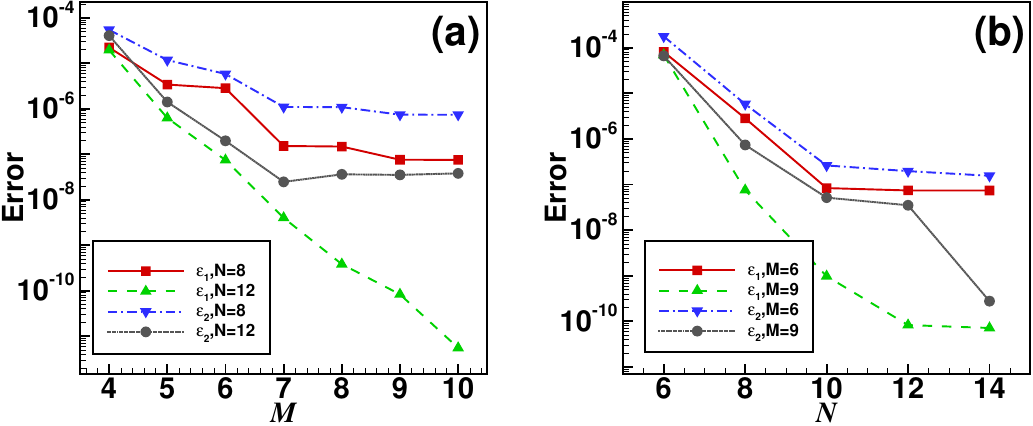}
	\caption{Spectrum error $\varepsilon_1$ and $\varepsilon_2$. (a): Error as function of $N$ for different $M$. (b): Error as function of $M$ for different $N$.}\label{fig331}
\end{figure}

Then we focus on the reduction part. Figure \ref{fig331} illustrates convergence as $N$ increases for $M=6$ and $9$, and as $M$ increases for $N=8$ and $12$, characterized by $\varepsilon_1$ and $\varepsilon_2$. Both panels demonstrate that $\varepsilon_1$ and $\varepsilon_2$ experience an exponential decrease as both $N$ and $M$ increase, ultimately reaching a stable value. Notably as shown in panel (a), the magnitude of this stable value diminishes with an increase in $M$. Even for relatively small values of $N$ and $M$ (e.g., $N=8$ and $M=6$), $\varepsilon_1$ and $\varepsilon_2$ are already smaller than $10^{-5}$, showcasing the high accuracy of the DF-RPM. Moreover, for small $M$ ($M\leq7$), the error curves of $N=8$ and $12$ decay exponentially, indicating that the error primarily stems from the basis reduction. Conversely, for larger $M$, the error curves of the two cases show significant differences, suggesting that the error is mainly attributed to the PM part. In panel (b), with an increasing $N$, both $\varepsilon_1$ and $\varepsilon_2$ decay exponentially, ultimately reaching a stable value. These numerical results are consistent with the theoretical analysis. In summary, the results exhibit high accuracy despite a substantial reduction in the number of bases.

We next conduct the study on the DOF required in the eigenvalue solver for the DF-RPM for $N=6, 8,10,12$ and $14$ with $M$ increasing from $4$ to $10$, and the results are summarized in Table \ref{tb2}. The DOF exhibits a quadratic decrease with respect to $M$. Theoretically, the DF-RPM of 3D systems has complexity $O(M^3)$ for given $N$,  and $O(N^6)$ for given $M$, while the complexity of the original DF-PM is $O(N^{12})$. The numerical results of Table \ref{tb2} are in agreement with these theoretical analysis. It also can be found that a small $M$ is able to reach high accuracy. {For $N=10$, the use of $M=8$ achieves an error level of $10^{-9}$. In this case, the DOF is $4.9$ times smaller than that of the DF-PM.} The reduction for larger $N$ will be more significant. {When $N=14$, the DOF reduction with $M=6$ becomes $17.2$ times, comparing the DF-RPM with the DF-PM, which also leads to a significant reduction in the CPU time.} It is noteworthy that for the DF-PM, even a modest $N$, such as $N=12$, results in a large DOF of $5971968$, rendering its memory and CPU time consumption unacceptable.

\begin{table}[htbp]
	\centering
	\fontsize{10}{10}\selectfont
	\begin{threeparttable}
		\caption{The DOF of the DF-RPM for different $N$ and $D$ in Example 2} 
		\label{tb2}
  \setlength{\tabcolsep}{6mm}
		\begin{tabular}{ccrrrrr}
			\toprule
			\midrule 
                  & $M$ & $N=6$ & $N=8$ & $N=10$ & $N=12$ & $N=14$\cr
			\midrule
             DF-PM & -&  93310 & 524286  &  1999998 & 5971966  &  15059070 \cr
			\midrule
			\multirow{7}{*}{DF-RPM} & 10 &  93310 & 431998  &  1143572 & 2249726  & 3813246  \cr
			& 9 &  85748 &  370384 &  913064 & 1714748  &  2809854 \cr
   			& 8 &  78606 &  297752 &  685998 & 1228248  &  1999998 \cr
      		& 7 &  65534 &  221182 &  476654 &  843748 & 1362942  \cr
         	& 6 & 48776  &  148174 &  314926 &  549248 &  877950 \cr
                & 5 &  31248 &  85748 &  194670 &  332748 &  524286 \cr
                & 4 & 15998  & 48776  & 101304  & 182248  &  281214 \cr
                \midrule
			\bottomrule
		\end{tabular}
	\end{threeparttable}
\end{table}

\subsection{Example 5}
In this example, we explore an aperiodic photonic system known as a moiré lattice \cite{wang2020localization, gao2023pythagoras}. Photonic moiré lattices arise from the superposition of two rotated square or hexagonal sublattices \cite{wang2020localization,huang2016localization}. Researchers are drawn to this system due to its intriguing electronic, optical, and magnetic properties, as well as its potential for designing novel materials with tailored functionalities \cite{andrei2021marvels,du2023moire}. In this system, the electric permittivity function is defined as \cite{efremidis2002discrete}:
\begin{equation}
    \varepsilon(\bm z)=(n_0+10^{-4}f(\bm z))^2,
\end{equation}
where $n_0=2.28$, $f(\bm z)=p_1v(\bm z)+p_2v(\mathbf{S}\bm z)$, $v(\bm z)=\cos(z_1)+\cos(z_2)+\cos(z_3)$ and the rotation matrix
\begin{equation*}
    \mathbf{S}=\begin{bmatrix}
        \cos\alpha\cos\gamma-\sin\alpha\cos\beta\sin\gamma & -\cos\alpha\sin\gamma-\sin\alpha\cos\beta\cos\gamma & \sin\alpha\sin\beta\\
        \sin\alpha\cos\gamma+\cos\alpha\cos\beta\sin\gamma & -\sin\alpha\sin\gamma+\cos\alpha\cos\beta\cos\gamma & -\cos\alpha\sin\beta\\
        \sin\beta\sin\gamma & \sin\beta\cos\gamma & \cos \beta
    \end{bmatrix}.
\end{equation*}
Here $(\alpha,\beta,\gamma)$ is the Euler angle between two sublattices. To the moir\'e lattices composed of two square lattices as considered here, they are periodic when $(\alpha,\beta,\gamma)$ all take Pythagorean angles and aperiodic otherwise \cite{huang2016localization}. Hence the projection matrix is 

\begin{equation*}
    \mathbf{P}=\begin{bmatrix}
        1  & 0 & 0 & \cos\alpha\cos\gamma-\sin\alpha\cos\beta\sin\gamma & \sin\alpha\cos\gamma+\cos\alpha\cos\beta\sin\gamma & \sin\beta\sin\gamma\\
        0 & 1 & 0 & -\cos\alpha\sin\gamma-\sin\alpha\cos\beta\cos\gamma & -\sin\alpha\sin\gamma+\cos\alpha\cos\beta\cos\gamma & \sin\beta\cos\gamma \\
        0 & 0 & 1 & \sin\alpha\sin\beta & -\cos\alpha\sin\beta & \cos \beta\\
    \end{bmatrix}.
\end{equation*}
 Throughout this example, the sublattices of the moir\'e systems have fixed parameter $p_1=p_2=1/6$ and $(\alpha,\beta,\gamma)=(\pi/6,\pi/6,\pi/6)$. We first show the eigenstates calculated by the DF-RPM. In Figure \ref{fig3-2}, we show one of the first normalized eigenstates, $\bm u=(u_1,u_2,u_3)$, under $N=10,M=6$. The field distribution and its specific iso-surfaces of the aperiodic moir\'e lattice with a rotation angle $(\pi/6,\pi/6,\pi/6)$ are presented.

\begin{figure}[h]
	\centering
	\includegraphics[width=1.00\textwidth]{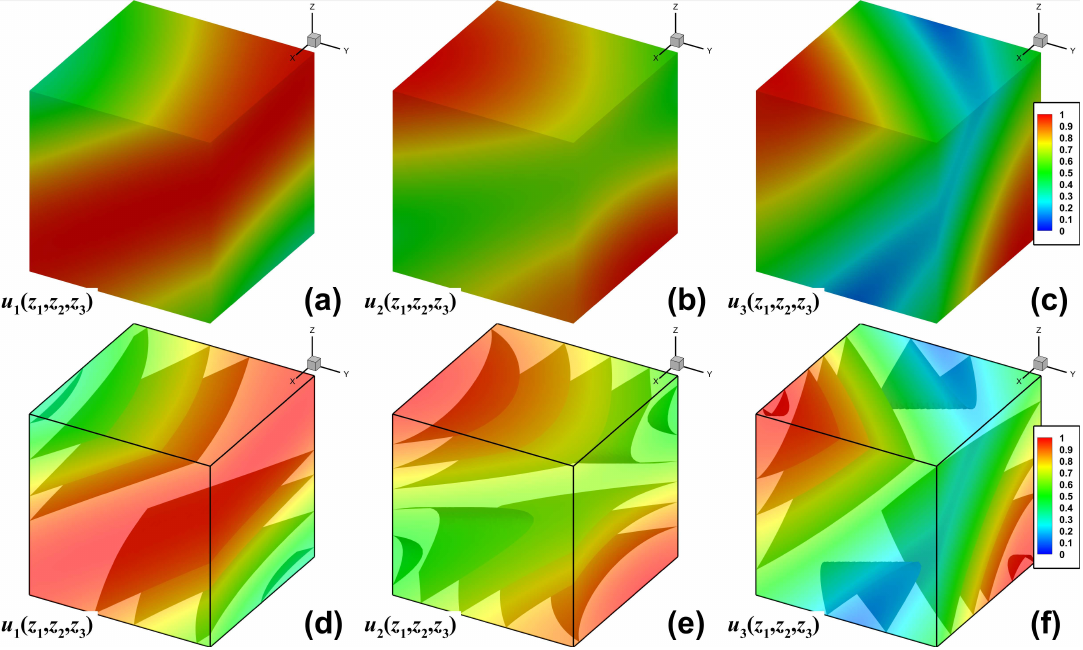}
	\caption{The normalized first eigenfunction $\bm u=(u_1,u_2,u_3)$ in three dimensions by the DF-RPM. $N=10,M=6$ is taken and the results in area $[-10,10]^2$. (a-c) the surface plot of the amplitude of $\bm u$. (d-f) the iso-surfaces.}\label{fig3-2}
\end{figure}

We proceeded to analyze the accuracy and convergence of the DF-RPM, and the results are displayed in Figure \ref{fig114}. In our computations, the "exact" results are obtained from the numerical results of the DF-RPM with $N=14$ and $M=8$. The absolute error of the first eigenvalue, denoted as $\varepsilon$, is approximately $0.0009$. Panel (a) presents the convergence with the increase of $M$ for $N=8$ and $12$. Exponential decay with $D$ is evident at the beginning, consistent with the previous error analysis. With an increasing $M$, the two curves differ significantly, indicating that the error at this point mainly arises from the DF-PM part. Overall, the accuracy with small $M$ (e.g., $D=4$) is sufficiently high to provide accurate solutions. These findings affirm that a significant reduction in basis functions can still maintain high accuracy.

Panel (b) illustrates the convergence of the numerical solution with the increase of $N$ for truncation coefficients $M=5, 6$ and $7$. $\varepsilon$ exhibits exponential decay with increasing $N$, ultimately converging to a fixed value dependent on $M$. Similar to Example 2, small values of $M$ yield highly accurate results. For instance, with $M=7$ and a slightly larger $N$, the RPM achieves accuracy at the level of $10^{-7}$ in the eigenvalue calculation.

\begin{figure}[h]
	\centering
	\includegraphics[width=0.8\textwidth]{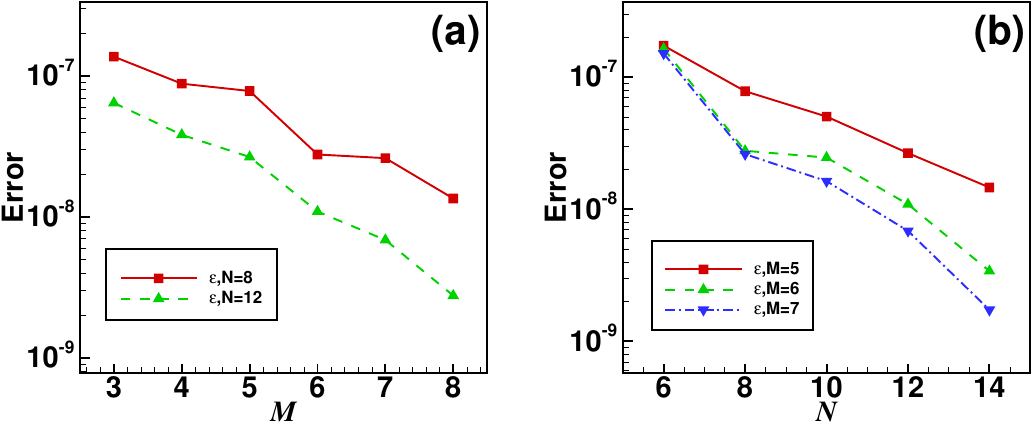}
	\caption{Errors of the first eigenvalue $\varepsilon$. (a): Error as function of $N$ for different $M$. (b): Error as function of $M$ for different $N$.}\label{fig114}
\end{figure}

\begin{table}[htbp]
	\centering
	\fontsize{10}{10}\selectfont
	\begin{threeparttable}
		\caption{The DOF of the DF-RPM for different $N$ and $M$ in Example 3} 
		\label{tb3}
  \setlength{\tabcolsep}{6mm}
		\begin{tabular}{ccrrrrr}
			\toprule
			\midrule 
                  & $M$ & $N=6$ & $N=8$ & $N=10$ & $N=12$ & $N=14$\cr
                  DF-PM & - & 93310  & 524286  &  1999998 & 5971966  &  15059070 \cr
			\midrule
			\midrule
   			\multirow{6}{*}{DF-RPM} & 8 & 93310 &  508158 & 1708314  & 4145462 &  8192876 \cr
      		& 7 & 92806 & 475870 & 1471966 & 3351376  &  6352886\cr
         	& 6 & 89710 & 414400 & 1167040  & 2505840  & 4574956\cr
                & 5 & 79990 &  322024 & 831154  & 1693200 &  2986880\cr
                & 4 & 61114 & 213768 & 510708  & 992940 &   1698188\cr
                & 3 & 36552 & 112398 & 251006  & 468406  &  781234\cr
                \midrule
			\bottomrule
		\end{tabular}
	\end{threeparttable}
\end{table}

Here we also display the DOF in the eigenvalue solver of the DF-RPM for $N=8,10,12$ and $14$ with $M$ increasing from $3$ to $8$. The results are similar with that of Example 2, which are summarized in Table \ref{tb3}. The numerical results of Table \ref{tb3} are in agreement with the theoretical analysis. It also can be found that a small $M$ is able to reach high accuracy. For $N=14$, the use of $M=3$ achieves an error level of {$5\times10^{-8}$}. In this case, the DOF is $19.3$ times smaller than that of the DF-PM, clearly illustrating the higher efficiency of our proposed method.

\section{Conclusion}\label{s6}
We introduced projection methods based on the divergence-free basis as a precise and efficient approach for quasicrystals in three dimensions. By establishing the de Rham complex, we constructed a point-wise divergence-free quasiperiodic Fourier spectral basis, then proposed an accurate and efficient projection scheme for the quasiperiodic source problem. Thanks to the observation of the decay rate of the divergence-free Fourier coefficients of eigenfunctions, we reduced the basis space and proposed a reduced scheme for the quasiperiodic Maxwell eigenvalue problem. The theoretical perspective is supported by an error bound for the approximation, highlighting the method's high accuracy. Numerical results in 3D problems underscore the algorithm's efficiency and accuracy, presenting appealing characteristics for diverse practical applications.

\section*{Acknowledge}
Z. G. and Z. X. are supported by the National Natural Science Foundation of China (NNSFC)(grants No. 12325113 and 12071288), Science and Technology Commission of Shanghai Municipality (grant No. 21JC1403700), and the HPC center of Shanghai Jiao Tong University. Z. Y. is supported by the NNSFC (No. 12101399) and the Shanghai Sailing Program (No. 21YF1421000).

\end{document}